\documentclass[12pt]{article} \textwidth=16.2cm \textheight=21cm
\usepackage{amsthm}
\usepackage{color}
\usepackage{cite}
\usepackage{graphicx}
\usepackage{psfrag}
\usepackage{url}
\usepackage{stfloats}
\usepackage{amsmath}
\usepackage{amssymb}
\usepackage{graphicx}
\usepackage[titletoc]{appendix}

\newtheorem{theorem}{Theorem}[section]

\newtheorem{lemma}{Lemma}[section]

\newtheorem{remark}{Remark}[section]

\def\a{\alpha} \def\b{\beta}

\def\dref#1{(\ref{#1})}

 \def\dfrac{\displaystyle\frac}
\newtheorem{corollary}{Corollary}[section]
\def\be{\begin{equation}}
\def\bel{\begin{equation}\label}
\def\ee{\end{equation}}
\def\ba{\begin{array}}
\def\ea{\end{array}}
\def\banl{\begin{eqnarray}\label}
\def\ean{\end{eqnarray}}

 \def\bna{\begin{eqnarray}}
\def\ena{\end{eqnarray}} \def\dref#1{(\ref{#1})}

\begin{document}

\title{When Adaptive Diffusion Algorithm Converges to True Parameter? \footnotemark[1]
\author{Zhaobo Liu\footnotemark[2]
, \,Chanying Li \footnotemark[2]
          }
}

\footnotetext[1]{The work was supported by  the National Natural Science Foundation of China under Grants 61422308 and 11688101.  }

\footnotetext[2]{%
L.~Liu and C.~Li are with the Key Laboratory of Systems and Control, Academy of Mathematics and Systems Science, Chinese Academy of Sciences, Beijing 100190, P.~R.~China.   They are also with the School of Mathematical Sciences, University of Chinese Academy of Sciences, Beijing 100049, P. R. China.  Corresponding author: Chanying~Li (Email: \texttt{cyli@amss.ac.cn}).}

 \maketitle

\begin{abstract}
We attempt to answer the question  what data brings   adaptive    diffusion    algorithms converging to true parameters.  The discussion begins with the diffusion recursive least squares (RLS).   When unknown parameters are scalar, the necessary and sufficient condition of  the convergence for the   diffusion RLS  is established,    in terms of the strong consistency and mean-square convergence both. However, for the general high dimensional parameter case, our results suggest that the  diffusion RLS  in a connected network might cause a  diverging  error,   even if  local data at every node could guarantee  the individual  RLS   tending to true parameters. Due to the possible failure of the diffusion RLS,
 we  prove that the diffusion Robbins-Monro (RM)   algorithm  could  achieve the  strong consistency and  mean-square convergence simultaneously,
under some cooperative information conditions. The convergence rates  of the diffusion RM    are  derived explicitly.

\end{abstract}

\section{Introduction}

Perhaps,
it is only natural  that this paper is intended to prove  adaptive diffusion algorithms   outperform their individual counterparts in terms of estimation performances. It is an accepted fact for the diffusion  least mean squares (LMS) with regard to mean stability and mean-square stability  (see \cite{LSayed08}, \cite{TuSayed12}, \cite{guo2018}). But this time, involving the sophisticated recursive least squares (RLS) in diffusion strategies, situation changes.

An adaptive network  is built up from a set of nodes which could communicate with their neighbors through interlinks. Each node observes partial information related to an unknown parameter of common interest     and performs local estimation separately. There are two main  types of fully decentralized strategies in distributed estimation, namely, consensus strategies  \cite{carli2008}, \cite{cws2014}, \cite{kar2011}, \cite{guoxie2018}  and diffusion strategies \cite{ab2016}, \cite{gha2013}, \cite{kha2012}, \cite{nos2015}, \cite{pig2015}, \cite{ta2010}.  In light of  local parameter estimation and      processed information sharing, the two networks enjoy a certain advantages in robustness and privacy. In particular,  compared with  individual  identification,  producing better estimates   in  collaborative  manners is very likely to be an absolute    cinch.  This guess was first proved false by \cite{TuSayed12}, since it found consensus networks can become unstable when all its nodes exhibit   stable behaviors in individual estimation processes. But at the same time, it showed that stability of the individual LMS  always infers stability of the diffusion LMS.
So, to some extent,   diffusion networks are more stable than  consensus ones. It was  confirmed again   in \cite{guo2018} recently by considering the normalized least mean squares (NLMS).
 Establishing a cooperative information condition, \cite{guo2018} concluded that the diffusion NLMS could track parameters effectively when none of the   local data provides sufficient information for individual identification. Almost all the existing literatures on the diffusion LMS-type algorithms suggest diffusion networks behave superiorly to non-cooperative schemes (see \cite{nos2015}, \cite{TuSayed12}, \cite{guo2018}).
Interestingly, as regard to the diffusion RLS, we cannot take it
for granted.



The diffusion RLS was proposed in \cite{Sayed08}, which  discussed a typical scenario attaining
bounded  mean-square errors. At  each node $i$,   the  data is required to be    independent and     tend towards steady that matrix $E P^{-1}_{k,i}$  becomes constant for all large time $k$. These constraints  are retained in other relevant  studies \cite{ba2011}, \cite{CLS2007}, \cite{Sayed08}, \cite{lll2014}, \cite{mat2012}, \cite{sayed2006} simply to make the problem   tractable.
However, for a variety of reasons,   connections between  data might be inevitable. More importantly,
\begin{eqnarray}\label{exci}
\lambda_{min}\left(P_{k,i}^{-1}\right)\rightarrow+\infty,
\end{eqnarray}
intuitively generates more informative  excitation signals than those for steady $P_{k,i}^{-1}$.
So what conclusions will survive, if the
 data utilized for estimation   admits no such  constraints?  Digging into this case, connections between the diffusion RLS and the non-cooperative RLS   are brought to the surface.

Indeed, for scalar unknown parameters,  the idea that cooperations among nodes  through diffusion networks help to promote estimation performances  is verified as expected here, the conclusion  for  high dimensional parameters  turns out to be quite different.  Opposite to  \cite{guo2018}, when parameters are vectors,   our results suggest that the convergence of the individual  RLS to true parameters at every node  cannot even guarantee the stability of
the diffusion RLS in a connected network, 
 let alone the identification task.

To be more precise, for a linear regression model with a scalar unknown parameter,     we find the necessary and sufficient condition on the regressor data, in a  cooperation form,  to guarantee the convergence of the diffusion RLS to the true parameter,   in the sense of the strong consistency and mean-square convergence. This critical condition degenerates to the necessary and sufficient condition of the    above two  convergences for the   individual  RLS,  when the underlying network  has  only one node. But  this  critical  convergence condition  can no longer be extended here in the high dimensional parameter case. Worse still, the  cooperation of the nodes   in a connected network might cause a  diverging  error   even \dref{exci} holds for every node $i$, which means the individual RLS  at each node, if is employed, tending to true parameters  \cite{hd76}, \cite{lai79}.
As a supplement, we   prove that the diffusion Robbins-Monro (RM)  algorithm could  achieve the  strong consistency and  mean-square convergence simultaneously,   when  regressor data  fails the diffusion RLS   for high dimensional   parameters. The two convergence rates of  the diffusion RM are explicitly derived. 

The rest of the paper is organized as follows. In the next
section, we  present the main theorems with the
proofs given in Sections \ref{gsta}--\ref{proofls2}.
The concluding remarks are included in Section
\ref{concluding}.

\section{   Main Results  }\label{MR}
Consider a  network consisting of $n$ nodes that trying to identify an unknown parameter in a collaborative manner.
At time $k$, each node $i$ observes 
a noisy signal $y_{k,i}\in \mathbb{R}$ and a  data signal  $\phi_{k,i}\in\mathbb{R}^{m}$. This process is
described by a stochastic linear regression model
\begin{equation}\label{model}
y_{k,i}=\theta^{\tau}\phi_{k,i}+\varepsilon_{k,i},\quad k\geq 0,\,\, i=1,\ldots,n,
\end{equation}
where $(\cdot)^{\tau}$ denotes the transpose operator, 
$\varepsilon_{k,i}$ is a scalar noise sequence and $\theta\in\mathbb{R}^{m}$ is an unknown  deterministic parameter.

Let  the network topology be           depicted   by a directed weighted graph $\mathcal{G}=(\mathcal{V},\mathcal{E},\mathcal{A})$, where $\mathcal{V}=\lbrace 1,2,\ldots,n\rbrace$ is the set of the nodes and $\mathcal{E}\subseteq \mathcal{V}\times \mathcal{V}$ is the set of the edges that any $(i,j)\in \mathcal{E}$  means $\mathcal{G}$ contains a directed path from $j$ to $i$.
The structure of the graph $\mathcal{G}$ is described
by  the weighted adjacency
matrix  $\mathcal{A}=\lbrace a_{ij}\rbrace_{n\times n}$, where $a_{ij}>0$ for  $(i,j)\in \mathcal{E}$ and $a_{ij}=0$ otherwise. 

We employ the adapt-then-combine (ATC) diffusion
strategy for the estimation algorithm, which is  recursively defined for each node $i$ by
\begin{enumerate}
\item  Adaption:
\begin{eqnarray}
\beta_{k+1,i}=\theta_{k,i}+L_{k,i}(y_{k,i}-\theta_{k,i}^{\tau}\phi_{k,i})\nonumber
\end{eqnarray}
with initial estimate $\theta_{0,i}\in \mathbb{R}^m$,   
where
$L_{k,i}\in\mathbb{R}^{m}$ is  to be designed  based on data $\phi_{0,i},\ldots,\phi_{k,i}$.
\item  Combination:
\begin{eqnarray}
\theta_{k+1,i}=\sum_{j=1}^{n}a_{ij}\beta_{k+1,i}.\nonumber
\end{eqnarray}
\end{enumerate}
Denote $\tilde\theta_{k,i}\triangleq\theta_{k,i}-\theta$, then
\begin{eqnarray}
\widetilde\Theta_{k+1}=(\mathcal{A}\otimes I_{m})(I_{mn}-F_{k})\widetilde\Theta_{k}+(\mathcal{A}\otimes I_{m})L_{k}V_{k},\nonumber
\end{eqnarray}
where
\begin{eqnarray}
\widetilde\Theta_{k}&\triangleq &{\mbox{col}\lbrace \tilde\theta_{k,1},\ldots,\tilde\theta_{k,n}\rbrace},\nonumber\\
 L_{k}&\triangleq &{\mbox{diag}\left\lbrace L_{k,1},\ldots,L_{k,n}\right\rbrace},\nonumber\\
\Phi_{k}&\triangleq &{\mbox{diag}\lbrace \phi_{k,1},\ldots,\phi_{k,n}\rbrace},\nonumber\\
 V_{k}&\triangleq &{\mbox{col}\lbrace\varepsilon_{k,1},\ldots,\varepsilon_{k,n}\rbrace},\nonumber\\
  F_{k}&\triangleq&{L_{k}\Phi_k^{\tau}}.
 \nonumber
\end{eqnarray}
Different $\lbrace L_{k,i}\rbrace$ result in variant types of adaptive algorithms, like the RLS, LMS and Kalman filtering. Since the parameter to be identified is time-invariant, we focus on the RLS and the Robbins-Monro algorithm.

\begin{remark}
Another well studied diffusion scheme  is  the combine-then-adapt (CTA) rule (see \cite{LSayed08}, \cite{TuSayed12}). Since the two strategies  are essentially the same for our problem, we only study the ATC diffusion strategy. All the results in this paper still hold for the CTA diffusion strategy.
\end{remark}


\subsection{Diffusion Recursive Least-Squares Algorithm}
In this section, we apply the RLS algorithm to estimate the unknown parameter $\theta$    based on   the ATC diffusion
strategy. That is,  $\lbrace L_{k,i}; k\geq 0, 1\leq i\leq n\rbrace$ are designed as
\begin{eqnarray}
\left\{
\begin{array}{l}
L_{k,i}=P_{k+1,i}\phi_{k,i}=\frac{P_{k,i}\phi_{k,i}}{1+\phi_{k,i}^{\tau}P_{k,i}\phi_{k,i}}\\
P_{k+1,i}^{-1}=I_m+\sum_{j=0}^{k}\phi_{j,i}\phi_{j,i}^{\tau}
\end{array}.
\right.\nonumber
\end{eqnarray}
\subsubsection{A Critical Convergence Theorem} We analyze  the estimation performance of the diffusion RLS algorithm  under
\begin{description}

\item[A1] $\mathcal{A}$ is an  irreducible and aperiodic doubly stochastic matrix with $\mathcal{A}^{\tau}\mathcal{A}$ being irreducible.
\item[A2]
The noises $\lbrace (\varepsilon_{k,1},\ldots,\varepsilon_{k,n})^{\tau}\rbrace_{k\geq 0}$ are mutually independent  and  for each $i=1,\ldots,n$,
$$
E\varepsilon_{k,i}=0,\, \forall k\geq 0\quad \mbox{and}\quad  \sup_{k\geq 1}E\varepsilon_{k,i}^2<M,
$$
where $M>0$ is a constant.

\item[A3] $ \phi_{k,i}, i=1,\ldots,n, k\geq 1$ are non-random constants.

\end{description}
\begin{remark}
If  graph $\mathcal{G}$ is undirected, connected
and containing a self-loop at each node, then it corresponds to a special case of  Assumption $A1$. See the  network topology of \cite{guo2018}.
\end{remark}

Recalling the well-known results  \cite[Theorem 1]{hd76} and \cite[Theorem 3.1]{lai79} on the least-squares (LS) estimator, we know that under Assumptions A2--A3,
for each single node $i$, if $\inf_{k\geq 0}E\varepsilon_{k,i}^2>0$, then
\begin{eqnarray*}
\hat\theta_{k,i}\stackrel{a.s.}{\longrightarrow}  \theta \quad \mbox{and}\quad E(  \hat \theta_{k,i}- \theta)^2\rightarrow 0
\end{eqnarray*}
are both equivalent to
\begin{eqnarray}\label{qiyi}
\lambda_{min}\left(P_{k,i}^{-1}\right)\rightarrow+\infty,
\end{eqnarray}
where  for  any initial $\hat\theta_{0,i}$ and  $k\geq 0$,
\begin{eqnarray}\label{iconvergence}
\hat\theta_{k+1,i}=\hat\theta_{k,i}+L_{k,i}(y_{k,i}-\phi_{k,i}^{\tau}\hat\theta_{k,i}).
\end{eqnarray}

Let $\|\cdot\|$ denotes the spectral norm of a matrix.
The two convergences are  now derived at every node     in a collaborative manner when the  unknown parameter  is a  scalar.


\begin{theorem}\label{ls1}
Let $m=1$.  Under Assumptions A1--A3,
 \begin{eqnarray}\label{lsconergence}
\|\widetilde\Theta_{k}\|\stackrel{a.s.}{\longrightarrow} 0\,\,\,\mbox{and}\,\,\, E\|\widetilde\Theta_{k}\|^2\rightarrow0\quad \mbox{as} \,\,k\rightarrow+\infty
\end{eqnarray}
 for any  initial $\Theta_0\in \mathbb{R}^n$, if and only if
   \begin{eqnarray}\label{Pinfi}
\lim_{k\rightarrow+\infty}\sum_{i=1}^nP_{k+1,i}^{-1}=+\infty.
\end{eqnarray}
\end{theorem}

\begin{remark}\label{r1}
(i) Discussions on the necessity  of Theorem \ref{ls1}:\\
(a) if \dref{Pinfi} fails,  as proved in Section \ref{gsta},  any initial values $\lbrace \theta_{0,i}, 1\leq i\leq n\rbrace$ except the ones satisfying $\sum_{i=1}^{n}\mu_i(\theta_{0,i}-\theta)=0$
will lead to
\begin{eqnarray}\label{fs}
\liminf_{k\rightarrow+\infty} E\|\widetilde\Theta_{k}\|^2>0\quad\mbox{and}\quad \|\widetilde\Theta_{k}\|\stackrel{p}{\nrightarrow} 0,\nonumber
\end{eqnarray}
where $\mu_1,\ldots,\mu_n> 0$ are some constants determined by data $\lbrace \phi_{k,i}\rbrace$ and matrix $\mathcal{A}$.\\
(b) when the  noises and data satisfy
\begin{eqnarray}\label{Ee>0}
\left\{
\begin{array}{ll}
E\varepsilon_{k,i}^2>0, \,\,\, \mbox{for all}\,\, k\geq 0, 1\leq i\leq n\\
E\varepsilon_{k,i}\varepsilon_{k,j}=0,\,\,\, \mbox{for all}\,\, k\geq 0, 1\leq i<j\leq n\\
\sum_{k=0}^{+\infty}\sum_{i=1}^{n}\phi_{k,i}^2\not=0
\end{array},
\right.
\end{eqnarray}
then given any initial $\Theta_0\in \mathbb{R}^n$ (including $\theta_{0,i}=\theta, i\in [1,n]$), \dref{lsconergence} is equivalent to \dref{Pinfi}. See  Appendix \ref{AppB}.\\
(ii) As for the sufficient part of Theorem \ref{ls1},    the convergence rate (see \dref{sumPThe})   of the estimation error satisfies
\begin{eqnarray}\label{ratels}
\sum_{k=0}^{+\infty}\sum_{i=1}^{n}(1-(P_{k+1,i}P_{k,i}^{-1})^2)\|\widetilde\Theta_{k}\|^2<+\infty,\quad\mbox{a.s.}\nonumber
\end{eqnarray}
with $\sum_{k=0}^{+\infty}\sum_{i=1}^{n}(1-(P_{k+1,i}P_{k,i}^{-1})^2)=+\infty$.\\
\end{remark}

We come to an analogous    conclusion on the strong consistency   of Theorem \ref{ls1}  when   data
$$\Phi_{k}= {\mbox{diag}\lbrace \phi_{k,1},\ldots,\phi_{k,n}\rbrace},\quad k\geq 0
$$
 is a random sequence.
Assume
\begin{description}

\item[A2']
$\lbrace (\varepsilon_{k,1},\ldots,\varepsilon_{k,n})^{\tau}\rbrace_{k\geq 0}$  are mutually independent and   there is a constant $M>0$  such that for all $i\in[1,n]$,
\begin{eqnarray*}
\left\{
\begin{array}{l}
E(\varepsilon_{k,i}|\Phi_{j}, 0\leq j\leq k)=0\\
\sup\limits_{k\geq 1}E(\varepsilon_{k,i}^2|\Phi_{j}, 0\leq j\leq k)\leq M,\quad \mbox{a.s.}.
\end{array}
\right.
\end{eqnarray*}
\end{description}

\begin{theorem}\label{lsr}
Under Assumptions A1 and A2', for any  initial $\Theta_0\in \mathbb{R}^n$, on set $\lbrace \lim_{k\rightarrow+\infty}\sum_{i=1}^nP_{k+1,i}^{-1}=+\infty\rbrace$,
\begin{eqnarray}\label{lsconergence1}
\|\widetilde\Theta_{k}\|\stackrel{a.s.}{\longrightarrow} 0,\quad \quad{as}\,\,k\rightarrow+\infty.\nonumber
\end{eqnarray}
\end{theorem}

The proof of Theorem \ref{lsr}   is similar to that of Theorem \ref{ls1} and given in Appendix \ref{AppB}. The above two theorems suggest   that   when the unknown parameter is a scalar,     the informative data of one single node is sufficient to guarantee the strong consistency (mean-square convergence) of the diffusion RLS via the connectivity of the underlying network. 



\subsubsection{  Diffusion Strategy Could Fail the Convergence}
When the unknown parameter is of high dimension, a little surprising result emerges, indicating that a diffusion strategy could play a destructive role, if the network topology is strongly connected:
\begin{description}
\item[A1'] $\mathcal{A}$ is irreducible.
\end{description}

\begin{theorem}\label{ls2'}
Let $m>1$ and  Assumptions A1'and A2 hold.   If  $\|E\widetilde\Theta_{0}\|\not=0$, then there is a series of data $\lbrace \Phi_{k}\rbrace_{k=0}^{+\infty}$ satisfying
\begin{eqnarray}\label{Pinfinity}
\lim_{k\rightarrow+\infty}\lambda_{min}(P_{k,i}^{-1})=+\infty,\quad i=1,\ldots,n,
\end{eqnarray}
such that $\sup_{k\geq 0}E\|\widetilde\Theta_{k}\|^2=+\infty$.
\end{theorem}

More divergences of the diffusion RLS occur,    if the noises in Assumption A2 
are specified  by
\begin{description}
\item[A2'']  $\lbrace (\varepsilon_{k,1},\ldots,\varepsilon_{k,n})^{\tau}\rbrace_{k\geq 0}$ is an  i.i.d random sequence with a multivariate normal distribution $N(0,\Sigma)$.

\end{description}

\begin{theorem}\label{ls2}
Let $m>1$ and  Assumptions A1' and  A2'' hold.   If  $\|E\widetilde\Theta_{0}\|\not=0$, then there is a series of data $\lbrace \Phi_{k}\rbrace_{k=0}^{+\infty}$ satisfying \dref{Pinfinity}
such that \\
(i) for some set $D_0$  with $P(D_0)>0$,
\begin{eqnarray}\label{ls2as}
\sup_{k\geq 0}\|\widetilde\Theta_{k}\|=+\infty,\quad\mbox{a.s. on}~D_0;
\end{eqnarray}
(ii) for any $\varepsilon>0$,
\begin{eqnarray}\label{ls2p}
\limsup_{k\rightarrow+\infty}P(\|\widetilde\Theta_{k}\|>\varepsilon)>0.\nonumber
\end{eqnarray}
\end{theorem}

\begin{remark}\label{randomthe}
Although parameter $\theta$ is modeled as a deterministic vector here, Theorems \ref{ls2'} still holds for random parameter $\theta$. Furthermore, if $\theta$ has a normal distribution, then
Theorem \ref{ls2} can be derived as well. See Section \ref{proofls2}.
\end{remark}

\begin{remark}
Let parameter $\theta$,  data  $\{ \phi_{k,i}\}$ and  noises $\{\varepsilon_{k,i}\}$ in model \dref{model}   all be random. If $\theta$, independent of $\{ \phi_{k,i}\}$,  is Gaussian distributed  and $\{\varepsilon_{k,i}\}$ possess the standard normal distributions, then
in view of \cite{St77},
 for each single node $i$ and any initial value $\hat\theta_{0,i}$,
\begin{eqnarray}\label{77}
\left\lbrace \lambda_{min}\left(P_{k,i}^{-1}\right)\rightarrow+\infty \right\rbrace\subset\left\lbrace\lim_{t\rightarrow+\infty}\hat\theta_{k,i}=\theta\right\rbrace,\nonumber
\end{eqnarray}
where $\hat\theta_{k,i}$ is the individual RLS defined by \dref{iconvergence}. So, Remark \ref{randomthe} means in stochastic framework,
the diffusion strategy still possibly do a disservice to estimation.
In this sense, we might need a  stronger condition to ensure the strong consistency of the diffusion RLS, compared with the individual case.
\end{remark}

\subsection{Diffusion Robbins-Monro Algorithm}\label{B}
Now, we are going to seek an adaptive algorithm competent for distributed estimation, no matter the parameter to be identified is a scalar or a vector.
The diffusion RM   is a suitable candidate.   It  achieves the  strong consistency and  mean-square convergence simultaneously,
under the cooperative information condition below:

\begin{description}

\item[A3'] There are two constants $c>0, \alpha\in[0,\frac{1}{2})$ such that
\begin{eqnarray}\label{a3}
\inf_{k\geq 1}k^{\alpha}\lambda_{min}\left(E\left[\frac{1}{n}\sum_{i=1}^{n}\frac{\phi_{k,i}\phi_{k,i}^{\tau}}{1+\|\phi_{k,i}\|^2}\bigg|\mathcal{F}_{k-1}\right]\right)>c,\nonumber
\end{eqnarray}
where $\mathcal{F}_{k}\triangleq\sigma\lbrace \phi_{j,i},\varepsilon_{j,i},0\leq j\leq k, 1\leq i\leq n\rbrace$.
\end{description}
Alternatively,  denoting $$\lambda_k(h)\triangleq\lambda_{min}\left(E\left[\frac{1}{nh}\sum_{i=1}^{n}\sum_{j=k}^{k+h-1}\frac{\phi_{j,i}\phi_{j,i}^{\tau}}{1+\|\phi_{j,i}\|^2}\bigg|\mathcal{F}_{k-1}\right]\right),$$
where $h$ is a fixed positive integer,
a more useful condition  is
\begin{description}
\item[A3''] Regressr $\{\phi_{k,i}\}$ satisfies\\
(i) for some $c>0, \alpha\in[0,\frac{1}{2})$ and $h\in \mathbb{N}^+$,
\begin{eqnarray}\label{a4}
\inf_{k\geq 1}k^{\alpha}\lambda_k(h)>c.
\end{eqnarray}
(ii) $\lbrace   \phi_{k,i}\rbrace$ is independent of  noises $\lbrace \varepsilon_{k,i}\rbrace$.
\end{description}

\begin{theorem}\label{lms}
Under Assumptions A1, A2' and A3'(or A3''), if the diffusion Robbins-Monro algorithm takes
\begin{eqnarray}
L_{k,i}=\frac{1}{(k+1)^{\beta}}\frac{\phi_{k,i}}{1+\|\phi_{k,i}\|^2},\quad k\geq 0,~i=1,\ldots,n,\nonumber
\end{eqnarray}
where $\beta\in(\frac{1}{2},1-\alpha)$, then\\
(i)
 as $k\rightarrow+\infty$,
\begin{eqnarray}
E\|\widetilde\Theta_{k}\|^2\rightarrow0 \quad \mbox{and} \quad   \|\widetilde\Theta_{k}\|\stackrel{a.s.}{\longrightarrow} 0;\nonumber
\end{eqnarray}
(ii) the mean-square convergence rate is
\begin{eqnarray}\label{cr}
\limsup_{k\rightarrow+\infty}k^{\beta-\alpha}E\|\widetilde\Theta_{k}\|^2\leq \frac{M}{sc},
\end{eqnarray}
where $s$   and $M$  are two constants  defined  in Lemma \ref{gjb2} and Assumption A2'.
In addition, if the noises further satisfy
\begin{eqnarray}\label{EVl}
\sup_{k\geq 1} E\left[(V_{k}^{\tau}V_{k})^{l}|\Phi_{j},  0\leq j\leq k\right]<+\infty, \quad \mbox{a.s.}
\end{eqnarray}
for some   $l> \frac{1}{\b-\a}$,  then  for any $\varepsilon \in (0,\b-\a-\frac{1}{l+1})$,
\begin{eqnarray}\label{asconv}
\|\widetilde\Theta_{k}\|^2=o(k^{-\varepsilon}),\quad \mbox{a.s.}.
\end{eqnarray}

\end{theorem}

\begin{remark}\label{lmsr}
(i)  In Theorem \ref{lms},      let $m=1$,   $\theta_{0,i}=0$,   $\phi_{k,i}=\sqrt{2c}$,    $\varepsilon_{k,i}=\varepsilon_{k,1}$, $E\varepsilon_{k,1}^2=M$ for all $k\geq 0$ and $i=1,\ldots,n$.
Then,   $\alpha=0$ in Assumption A3' and  $\widetilde\Theta_{k}=\textbf{1}\cdot \tilde\theta_{k,1}$ with
\begin{eqnarray*}
E\widetilde\Theta_{k+1}^{\tau}\widetilde\Theta_{k+1}
&=&E\left[\widetilde\Theta_{k}^{\tau}(I_{mn}-F_{k})(\mathcal{B}\otimes I_{m})(I_{mn}-F_{k})\widetilde\Theta_{k}\right]\quad\nonumber\\
&&+tr(E[(\mathcal{A}\otimes I_{m})L_{k}V_{k}V_{k}^{\tau}L_k^{\tau}(\mathcal{A}^{\tau}\otimes I_{m})])\nonumber\\
&=&\left(1-\frac{2c}{(1+2c)(k+1)^{\beta}}\right)^2 E\widetilde\Theta_{k}^{\tau}\widetilde\Theta_{k} +\frac{2cnM}{(1+2c)^2(k+1)^{2\beta}}\nonumber\\
&\geq &\left(1-\frac{4c}{(1+2c)(k+1)^{\beta}}\right) E\widetilde\Theta_{k}^{\tau}\widetilde\Theta_{k}+\frac{2cnM}{(1+2c)^2(k+1)^{2\beta}},
\end{eqnarray*}
which by \cite[Lemma 4.2]{fa1967}   yields
\begin{eqnarray*}
\liminf_{k\rightarrow+\infty}k^{\beta}E\|\widetilde\Theta_{k}\|^2\geq \frac{nM}{2+4c}.
\end{eqnarray*}
So, generally speaking, the  order of magnitude of the convergence rate   in \dref{cr} can not    be improved if no further conditions are imposed.
\\
(ii) By \dref{cr},    constant $s$  is important to the  performance of the  mean-square convergence for the diffusion RM.
Note that if $\mathcal{A}$ is symmetric and $\inf_{i\in[1,n]}a_{ii}>0$, an analogous proof of  \cite[Lemma 5.10]{guo2018} shows that in  Lemma \ref{gjb2}, we can  select
\begin{eqnarray}
s=\frac{\inf_{i\in[1,n]}a_{ii}}{32n(1+4h)^2}\lambda(\mathcal{G}),\nonumber
\end{eqnarray}
 where $\lambda(\mathcal{G})$ is the smallest positive eigenvalue of the Laplacian matrix $I_{n}-\mathcal{A}$ and $h$ is defined in Assumption A3''. See Appendix \ref{AppC}. By \textit{Cheeger's inequality} \cite{chung1996}, $\lambda(\mathcal{G})\geq h_{\mathcal{G}}^2/2$, where
$h_{\mathcal{G}}$  is the Cheeger constant that  describes the difficulty of breaking the connectivity of $\mathcal{G}$.
Rewrite  \dref{cr}  as
\begin{eqnarray}
\limsup_{k\rightarrow+\infty}k^{\beta-\alpha}\frac{\sum_{i=1}^{n}E\|\tilde\theta_{k,i}\|^2}{n}\leq \frac{M}{snc},\nonumber
\end{eqnarray}
then
\begin{eqnarray}
\limsup_{k\rightarrow+\infty}k^{\beta-\alpha}\frac{\sum_{i=1}^{n}E\|\tilde\theta_{k,i}\|^2}{n}\leq \frac{64(1+4h)^2M}{ch_{\mathcal{G}}^2\inf_{i\in[1,n]}a_{ii}}.\nonumber
\end{eqnarray}
So,  for symmetric $\mathcal{A}$ with   $\inf_{i\in[1,n]}a_{ii}>0$,  the convergence performance of the diffusion RM  could be enhanced by  promoting  the connectivity of $\mathcal{G}$.
\end{remark}

\begin{remark}
To better understand the problem, we compare the diffusion RM with the diffusion RLS and the diffusion NMLS. \\
(i)  It is easy to  verify that data $\{\phi_{k,i}\}$ constructed in Section \ref{proofls2} satisfies Assumption A3''.   So, for  high dimensional parameters, even if $\{\phi_{k,i}\}$ corresponds to a diverging error of the diffusion RLS, it  still  stands a chance to generate   estimates converging to  true parameters, by applying the diffusion RM.   \\
(ii) The  cooperative information condition derived in  \cite{guo2018} requires
$\lbrace \lambda_{k}, k\geq 0\rbrace\in S^{0}(\lambda)$, where  $\lambda\in(0,1)$ and
\begin{eqnarray}
S^{0}(\lambda)\triangleq\bigg\lbrace \lbrace a_{k}\rbrace: a_{k}\in[0,1], E\bigg[\prod_{j=i+1}^{k}(1-a_j)\bigg]\leq K\lambda^{k-i}, \forall k>i, i\geq 0, ~\mbox{for~some~}K>0 \bigg\rbrace.\nonumber
\end{eqnarray}
Note that this  cooperative information condition  is necessary and sufficient for the stability of the diffusion NLMS algorithm in \cite{guo2018},  whenever $\{\phi_{k,i}\}$ is $\phi$-mixing. However, by \cite[Theorem 2.3]{guo1994}, $\lbrace \lambda_{k}, k\geq 0\rbrace\in S^{0}(\lambda)$ implies \dref{a4} with $\alpha=0$ 
 for any $\phi$-mixing data $\{\phi_{k,i}\}$. 
So, the diffusion RM  could deal with some data beyond the capability of the diffusion NLMS, as far as the time-invariant-parameter case is concerned.

\end{remark}

\section{Proof of Theorem \ref{ls1} }\label{gsta}

We preface the proof with a simple lemma below.

\begin{lemma}\label{bkl}
Let  $\{e_k\}$ be a series of nonnegative real numbers.\\
(i) If for some $d_k\geq 0$ and $\sum_{k=0}^{+\infty}d_k<+\infty$,
\begin{eqnarray}
e_{k+1}\leq e_k+d_k,\quad \forall k\geq 0,\nonumber
\end{eqnarray}
then $\lim_{k\rightarrow+\infty}e_k$ exists.\\
\noindent(ii) If there exist two nonnegative sequences $\{a_k\}$ and $\{b_k\}$ with $\sum_{k=0}^{+\infty}a_k=+\infty$ and $\sum_{k=0}^{+\infty}b_k<+\infty$ such that
\begin{eqnarray}
e_{k+1}\leq(1-a_k)e_k+b_k,\quad \forall k\geq 0,\nonumber
\end{eqnarray}
 then
$\lim_{k\rightarrow+\infty}e_k=0.$
\end{lemma}

\begin{proof}
(i) Fix an integer $k>0$. Then, for any $l\geq k$,
\begin{eqnarray}
e_{l}\leq e_{k}+\sum_{i=k}^{l-1}d_i\leq e_{k}+\xi_{k},\nonumber
\end{eqnarray}
where $\xi_{k}=\sum_{i=k}^{+\infty}d_i$. So,
\begin{eqnarray}
e_{k}\geq \limsup_{l\rightarrow+\infty}e_{l}-\xi_{k},\nonumber
\end{eqnarray}
which together with $\lim_{k\rightarrow+\infty}\xi_{k}=0$ yields
\begin{eqnarray}
\liminf_{k\rightarrow+\infty}e_{k}\geq \limsup_{l\rightarrow+\infty}e_{l}-\lim_{k\rightarrow+\infty}\xi_{k}=\limsup_{l\rightarrow+\infty}e_{l}.\nonumber
\end{eqnarray}
Then,  $\lim_{k\rightarrow+\infty}e_k$ exists.

To prove (ii), note that $e_{k+1}\leq e_k+b_k$, where $\sum_{k=0}^{+\infty}b_k<+\infty$. Therefore,  $\lim_{k\rightarrow+\infty}e_k$ exists by (i).
Suppose $e\triangleq\lim_{k\rightarrow+\infty}e_k>0$, so there is a $N>0$ such that $e_{k}>\frac{e}{2}$    for all $k>N$.
Consequently,
\begin{eqnarray*}\label{eNi}
e_{N+i}-e_{N+1}&=&\sum_{k=N+1}^{N+i-1}(e_{k+1}-e_{k})\leq  -\sum_{k=N+1}^{N+i-1}a_ke_{k}+\sum_{k=N+1}^{N+i-1}b_k\leq -\frac{a}{2}\sum_{k=N+1}^{N+i-1}a_{k}+\sum_{k=N+1}^{+\infty}b_k,
\end{eqnarray*}
which shows  $e_{N+i}\rightarrow -\infty$ by letting $i\rightarrow+\infty$. This leads to a contradiction and hence $e=0$.
\end{proof}

\begin{lemma}\label{ty}
Let $\lbrace e_{k},k\geq 0\rbrace$ and $\lbrace d_{k},k\geq 0\rbrace$ be   two non-negative processes   adapted to a filtration $\lbrace \mathcal{G}_{k},k\geq 0\rbrace$. If
\begin{eqnarray}
E[e_{k+1}|\mathcal{G}_{k}]\leq e_{k}+b_k-d_k,\quad k\geq 0\nonumber
\end{eqnarray}
for some  $b_k\geq 0$ with $\sum_{k=0}^{+\infty}b_k<+\infty$,  then
\begin{eqnarray}\label{dk<inf}
\sum_{k=0}^{+\infty}d_k<+\infty \quad  \mbox{a.s.}.
\end{eqnarray}
In addition,  if $\lim_{k\rightarrow+\infty}Ee_{k}=0$, then
\begin{eqnarray*}
\lim_{k\rightarrow+\infty}e_{k}=0,\quad \mbox{a.s.}.
\end{eqnarray*}
\end{lemma}

\begin{proof}
As a matter of fact, \dref{dk<inf} is a direct result of \cite[Lemma 1.2.2]{ch02} and this lemma further shows that
there exists a random variable $e_{\infty}$ such that $E|e_{\infty}|<+\infty$ and
\begin{eqnarray}
\lim_{k\rightarrow+\infty}e_{k}=e_{\infty},\quad\mbox{a.s.}.\nonumber
\end{eqnarray}
Since $e_k\geq 0$, by \textit{Fatou's lemma},
\begin{eqnarray}
0=\liminf_{k\rightarrow+\infty}Ee_{k}\geq Ee_{\infty},\nonumber
\end{eqnarray}
which indicates $e_{\infty}=0$ almost surely.
\end{proof}

Fix an integer $h\geq 1$. Let $\lbrace A_{k,i};k=1,\ldots,h, i=1,\ldots,n\rbrace$ be a sequence of $m\times m$ symmetric random matrices     satisfying $0\leq A_{k,i}\leq I_m$. Denote $I_{k}(A)\triangleq\mbox{diag}\lbrace A_{k,1},\ldots,A_{k,n}\rbrace$ and
 \begin{eqnarray}\label{ap}
\left\{
\begin{array}{l}
\psi_{0}\triangleq I_{mn}\\
\psi_{k}\triangleq\prod_{j=k}^{1}((\mathcal{A}\otimes I_{m})I_{j}(A)),\quad k=1,\ldots,h.
\end{array}
\right.
\end{eqnarray}
 The following lemma shows

\begin{lemma}\label{gjb2}
Under Assumption A1, for any $\sigma$-algebra $\mathcal{F}$,
there is a constant $s\in(0,1)$ determined by $h$ and $\mathcal{A}$ such that
\begin{eqnarray}\label{sbd12}
\lambda_{min}(E\left[I_{mn}-\psi_h^{\tau}\psi_h\big|\mathcal{F}\right])\geq s\lambda_{min}\left(E\left[\sum_{k=1}^{h}\sum_{i=1}^{n}(I_{m}-A_{k,i}^2)\bigg|\mathcal{F}\right]\right).\nonumber
\end{eqnarray}
\end{lemma}

\begin{proof}
Denote $\mathcal{B}\triangleq{\mathcal{A}^{\tau}\mathcal{A}}$.   Since $\mathcal{B}$ is irreducible, for any $i\in[1,n-1]$, there is an integer  $d_i\geq 2$ and some distinct $c_{1}^{i},\ldots,c_{d_i}^{i}\in [1,n]$ such that
 \begin{eqnarray*}
\left\{
\begin{array}{l}
c_{1}^{i}=i,\quad c_{d_i}^{i}=i+1\\
\mathcal{B}[c_j^{i},c_{j+1}^{i}]>0,\quad j\in[1,d_i-1]
\end{array},
\right.
\end{eqnarray*}
where $M[i,j]$  refers to the $(i,j)$th  entry of a matrix $M$.
Let $q\triangleq\sum_{i=1}^{n-1}d_i-(n-2)$ and  define a sequence of $b_j,j=1,\ldots,q$ with   $b_1=c_1^{1}$ and
$b_j=c_{j-\sum_{i=1}^{l}(d_i-1)}^{l+1},$
where  $l\in[0,n-2]$ and
$$1+\sum_{i=1}^{l}(d_i-1)< j\leq 1+\sum_{i=1}^{l+1}(d_i-1).$$
Hence $\mathcal{B}[b_j,b_{j+1}]>0$ for all $j\in[1,q-1]$.

Select
\begin{eqnarray}
0<s< \frac{\min_{j\in[1,q-1]}\mathcal{B}[b_j,b_{j+1}]}{512h^3n^4q(1+n^2)}\nonumber
\end{eqnarray}
and denote
\begin{eqnarray}
\rho\triangleq\lambda_{min}\left(E\left[\sum_{k=1}^{h}\sum_{i=1}^{n}(I_{m}-A_{k,i}^2)\bigg|\mathcal{F}\right]\right).\nonumber
\end{eqnarray}
Now, suppose for a constant vector $ x\in\mathbb{R}^{mn}$ with $\|x\|=1$,
\begin{eqnarray}\label{fz}
x^{\tau}E\left[I_{mn}-\psi_{h}^{\tau}\psi_{h}\big|\mathcal{F}\right]x<s\rho
\end{eqnarray}
on some trajectory. We prove that on this trajectory,  for any $ k\in [1,h]$,
\begin{eqnarray}\label{ind}
E[\|\psi_k x-x\|^2|\mathcal{F}]<\frac{\rho}{64hn}.
\end{eqnarray}
To this end,  write $\psi_k x=\mbox{col}\lbrace z_{k,1},\ldots,z_{k,n}\rbrace\in\mathbb{R}^{mn}$, $k\in[0,h]$.
Observe that
\begin{eqnarray*}
x^{\tau}(I_{mn}-\psi_{k+1}^{\tau}\psi_{k+1})x=  ( \psi_k x)^\tau    \left( I_{mn}- I_{k+1}(A) (\mathcal{B} \otimes I_{m})I_{k+1}(A) \right )           ( \psi_k x )+ x^\tau (I_{mn}- \psi_k ^\tau \psi_k )x  ,
\end{eqnarray*}
a direct calculation yields
\begin{eqnarray}\label{bds13}
x^{\tau}E\left[I_{mn}-\psi_{h}^{\tau}\psi_{h}\big|\mathcal{F}\right]x
&= & \sum_{1\leq i<j\leq n}\mathcal{B}[i,j] \sum_{k=0}^{h-1}E[\|A_{k+1,i}z_{k,i}-A_{k+1,j}z_{k,j}\|^2|\mathcal{F}]  \nonumber\\
&&+\sum_{k=0}^{h-1}x^{\tau}E[\psi_{k}^{\tau}(I_{mn}-I_{k+1}^2(A))\psi_{k}|\mathcal{F}]x,
\end{eqnarray}
 which, together with \dref{fz}, implies that for any $i\in[1,n-1]$,
\begin{eqnarray}
\sum_{j=1}^{d_{i}-1}E\left[\left\|A_{k+1,c_j^{i}}z_{k,c_j^{i}}-A_{k+1,c_{j+1}^{i}}z_{k,c_{j+1}^{i}}\right\|^2\bigg|\mathcal{F}\right]<\frac{s\rho}{\min_{j\in[1,d_i-1]}\mathcal{B}[c_j^{i},c_{j+1}^{i}]},\nonumber
\end{eqnarray}
and hence
\begin{eqnarray}\label{lh12}
\sum_{j=1}^{q-1}E\left[\left\|A_{k+1,b_j}z_{k,b_j}-A_{k+1,b_{j+1}}z_{k,b_{j+1}}\right\|^2\bigg|\mathcal{F}\right]
< \frac{ns\rho}{\min_{j\in[1,q-1]}\mathcal{B}[b_j,b_{j+1}]}.
\end{eqnarray}
By \dref{lh12} and  \textit{Cauchy-Schwarz inequality},
\begin{eqnarray}\label{eb112}
E\left[\|A_{k+1,i}z_{k,i}-A_{k+1,j}z_{k,j}\|^2\bigg|\mathcal{F}\right]
<\frac{qns\rho}{\min_{j\in[1,q-1]}\mathcal{B}[b_j,b_{j+1}]},\quad \forall  i\not=j.
\end{eqnarray}

Furthermore, since
\begin{eqnarray*}
z_{k,i}^{\tau}(I_{m}-A_{k+1,i}^2)z_{k,i}+z_{k,j}^{\tau}(I_{m}-A_{k+1,j}^2)z_{k,j}
\geq \frac{1}{2}\|(I_m-A_{k+1,i})z_{k,i}-(I_m-A_{k+1,j})z_{k,j}\|^2,
\end{eqnarray*}
\dref{fz} and \dref{bds13}  imply
\begin{eqnarray*}
 \frac{1}{2}\max_{i,j}E[\|(I_m-A_{k+1,i})z_{k,i}-(I_m-A_{k+1,j})z_{k,j}\|^2|\mathcal{F}]
 \leq \sum_{i=1}^{n}E[z_{k,i}^{\tau}(I_{m}-A_{k+1,i}^2)z_{k,i}|\mathcal{F}]<s\rho
\end{eqnarray*}
and
\begin{eqnarray}
E[\|I_{k+1}(A)\psi_{k}x-\psi_{k}x\|^2|\mathcal{F}]
&=&\sum_{i=1}^{n}E[z_{k,i}^{\tau}(I_{m}-A_{k+1,i})^2z_{k,i}|\mathcal{F}]\nonumber\\
&\leq &\sum_{i=1}^{n}E[z_{k,i}^{\tau}(I_{m}-A_{k+1,i}^2)z_{k,i}|\mathcal{F}]<s\rho.\nonumber
\end{eqnarray}
So,
\begin{eqnarray}\label{29}
E[\|\psi_{k+1}x-(\mathcal{A}\otimes I_m)\psi_k x\|^2|\mathcal{F}]
&=&E[\|(\mathcal{A}\otimes I_m)(I_{k+1}(A)\psi_{k}x-\psi_{k}x)\|^2|\mathcal{F}]\nonumber\\
&\leq &E[\|I_{k+1}(A)\psi_{k}x-\psi_{k}x\|^2|\mathcal{F}]<s\rho
\end{eqnarray}
and
\begin{eqnarray}\label{sp}
&&\sum_{j=1}^{q-1} \mathcal{B}[b_j,b_{j+1}]  E[\|z_{k,b_j}-z_{k,b_{j+1}}\|^2|\mathcal{F}]\leq
\sum_{1\leq i<j\leq n}\mathcal{B}[i,j] E[\|z_{k,i}-z_{k,j}\|^2|\mathcal{F}]\nonumber\\
&\leq & 2\sum_{1\leq i<j\leq n}\mathcal{B}[i,j] (E[\|A_{k+1,i}z_{k,i}-A_{k+1,j}z_{k,j}\|^2+\|(I_m-A_{k+1,i})z_{k,i}-(I_m-A_{k+1,j})z_{k,j}\|^2|\mathcal{F}]) \nonumber\\
&\leq& 2s\rho+2n^2s\rho.
\end{eqnarray}
Similar to \dref{eb112}, by \textit{Cauchy-Schwarz inequality} and  \dref{sp},
\begin{eqnarray}\label{zij}
E[\|z_{k,i}-z_{k,j}\|^2|\mathcal{F}]&\leq &q\sum_{j=1}^{q-1} E[\|z_{k,b_j}-z_{k,b_{j+1}}\|^2|\mathcal{F}]< \frac{2qns\rho(1+n^2)}{\min_{j\in[1,q-1]}\mathcal{B}[b_j,b_{j+1}]},\nonumber\\
&<&\frac{\rho}{256h^3n^3},\quad    i<j.
\end{eqnarray}
Since $\mathcal{A}$ is a stochastic matrix,
\begin{eqnarray}
E[\|\psi_k x-(\mathcal{A}\otimes I_m)\psi_k x\|^2|\mathcal{F}]
=\sum_{i=1}^{n}E\bigg[\bigg\|\sum_{j=1}^{n}\mathcal{A}[i,j](z_{k,i}-z_{k,j})\bigg\|^2\bigg|\mathcal{F}\bigg]<\frac{\rho}{256h^3n},\nonumber
\end{eqnarray}
which together with \dref{29} leads to
\begin{eqnarray}\label{kk}
E[\|\psi_{k+1}x-\psi_k x\|^2|\mathcal{F}]
&\leq & 2E[\|\psi_{k+1}x-(\mathcal{A}\otimes I_m)\psi_k x\|^2|\mathcal{F}]+2E[\|\psi_k x-(\mathcal{A}\otimes I_m)\psi_k x\|^2|\mathcal{F}]\nonumber\\
&<&\frac{\rho}{64h^3n}.
\end{eqnarray}
Note that \dref{kk} holds for $k=0,\ldots,h-1$, by  \textit{ Cauchy-Schwarz inequality} again, for $k\in [1,h]$,
\begin{eqnarray*}
E[\|\psi_k x-x\|^2|\mathcal{F}]&\leq& k\sum_{j=1}^{k}E[\|\psi_j x-\psi_{j-1} x\|^2|\mathcal{F}]<h^2\left(\frac{\rho}{64h^3n}\right)=\frac{\rho}{64hn},
\end{eqnarray*}
which is exactly \dref{ind}.

 Now, let $k=0$ in \dref{zij}, it yields that
 $\|x_1-x_i\|^2<\frac{\rho}{16hn^2}$  for all $i>1$.
   Since $\sum_{i=1}^{n}\|x_i\|^2=1$,
    $$\|x_1\|^2\geq \frac{1}{2n-1}-\frac{1}{16hn^2}\rho>\frac{1}{4n}.$$
  Moreover, by applying \textit{Cauchy-Schwarz inequality},
 \begin{eqnarray}
&&\frac{1}{2}x^{\tau}(I_{mn}-I_{k+1}^2(A))x\nonumber\\
&\leq &x^{\tau}\psi_{k}^{\tau}(I_{mn}-I_{k+1}^2(A))\psi_{k}x+(\psi_{k+1}x- x)^{\tau}(I_{mn}-I_{k+1}^2(A))(\psi_{k+1}x- x)\nonumber\\
&\leq &x^{\tau}\psi_{k}^{\tau}(I_{mn}-I_{k+1}^2(A))\psi_{k}x +\|\psi_{k+1}x- x\|^2,\nonumber
\end{eqnarray}
therefore,
\begin{eqnarray}
x^{\tau}E\left[I_{mn}-\psi_{h}^{\tau}\psi_{h}\big|\mathcal{F}\right]x
&\geq &\sum_{k=0}^{h-1}x^{\tau}E[\psi_{k}^{\tau}(I_{mn}-I_{k+1}^2(A))\psi_{k}|\mathcal{F}]x\nonumber\\
&\geq &\frac{1}{2}\sum_{k=0}^{h-1}x^{\tau}E[(I_{mn}-I_{k+1}^2(A))|\mathcal{F}]x-\sum_{k=0}^{h-1}E[\|\psi_{k+1}x- x\|^2|\mathcal{F}]\nonumber\\
&=&\frac{1}{2}\sum_{i=1}^{n}x_i^{\tau}E\left[\sum_{k=0}^{h-1}(I_{m}-A_{k+1,i}^2)\bigg|\mathcal{F}\right]x_i-\sum_{k=0}^{h-1}E[\|\psi_{k+1}x- x\|^2|\mathcal{F}]\nonumber\\
&\geq &\frac{1}{4}x_1^{\tau}E\left[\sum_{k=0}^{h-1}\sum_{i=1}^{n}(I_{m}-A_{k+1,i}^2)\bigg|\mathcal{F}\right]x_1-\frac{h}{2}\sum_{i=2}^{n}\|x_1-x_i\|^2-\frac{\rho}{64n}\nonumber\\
&\geq & \frac{\rho}{16n}-\frac{\rho}{32n}-\frac{\rho}{64n}>s\rho,\nonumber
\end{eqnarray}
which contradicts to \dref{fz}. So, on every trajectory,
 $$x^{\tau}E\left[I_{mn}-\psi_{h}^{\tau}\psi_{h}\big|\mathcal{F}\right]x\geq s\rho$$
  holds for all unit vector $x\in\mathbb{R}^{mn}$ and  Lemma \ref{gjb2} follows.
\end{proof}

Taking $m=1$ and $h=1$ 
in Lemma \ref{gjb2} gives
\begin{corollary}\label{lsl}
Let $c=(c_1,\ldots,c_n)^{\tau}\in\mathbb{R}^{n}$ be a sequence of random variables satisfying $ c_i\in[0,1]$ for all $i\in[1,n]$. Denote $I(c)\triangleq\mbox{diag}\lbrace c_1,\ldots,c_n\rbrace$, then there is a constant $s\in(0,1)$ depending on $\mathcal{A}$   such that
\begin{eqnarray}
\lambda_{max}(I(c)\mathcal{A}^{\tau}\mathcal{A}I(c))\leq 1-s\sum_{i=1}^n(1-c_i^2).\nonumber
\end{eqnarray}
\end{corollary}

\begin{proof}[Proof of Theorem \ref{ls1}]
First, we show the sufficiency.  Without loss of generality,  assume $\lim\limits_{k\rightarrow+\infty}P_{k+1,1}^{-1}=+\infty$.
Since $m=1$,
\begin{eqnarray}\label{1w}
\widetilde\Theta_{k+1}=\mathcal{A}(I-F_k)\widetilde\Theta_{k}+\mathcal{A}L_{k}V_{k}.
\end{eqnarray}
Denoting $\Lambda_k=E\left[\widetilde\Theta_{k}\widetilde\Theta^{\tau}_{k}\right]$, Assumption A3 shows
\begin{eqnarray}\label{jj}
\Lambda_{k+1}=\mathcal{A}(I-F_k)\Lambda_k(I-F_k)\mathcal{A}^{\tau}+\mathcal{A}L_{k}E[V_{k}V_{k}^{\tau}] L_{k} \mathcal{A}^{\tau}.\,\,\,
\end{eqnarray}
In view of Assumption A2, 
 applying  \emph{Neumann inequality} and   Corollary \ref{lsl} leads to
\begin{eqnarray}\label{wss}
\mbox{tr}(\Lambda_{k+1})&\leq &\left(1-s\sum_{i=1}^{n}(1-(P_{k+1,i}P_{k,i}^{-1})^2)\right)\mbox{tr}(\Lambda_{k})+nM\sum_{i=1}^{n}P_{k+1,i}^2\phi_{k,i}^2.
\end{eqnarray}
Because $\lim_{k\rightarrow+\infty}P_{k+1,1}^{-1}=+\infty$,
$$\prod_{k=0}^{+\infty}(1-(1-(P_{k+1,1}P_{k,1}^{-1})^2))=\prod_{k=0}^{+\infty} (P_{k+1,1}P_{k,1}^{-1})^2   =0,$$
 which infers
\begin{eqnarray}\label{sumPphiinf}
\sum_{k=0}^{+\infty}(1-(P_{k+1,1}P_{k,1}^{-1})^2)=+\infty.
\end{eqnarray}
Furthermore,
\begin{eqnarray}\label{wsj}
\sum_{i=1}^{n}\sum_{k=0}^{+\infty}P_{k+1,i}^2\phi_{k,i}^2&=&\sum_{i=1}^{n}\sum_{k=0}^{+\infty}(1-P_{k+1,i}P^{-1}_{k,i})P_{k+1,i}\nonumber\\
&< &   \sum_{i=1}^{n}\sum_{k=0}^{+\infty} \left( P^{-1}_{k+1,i}P_{k,i}-1      \right)P_{k+1,i}\leq \sum_{i=1}^{n}P_{0,i}<+\infty,
\end{eqnarray}
 we thus conclude $\lim_{k\rightarrow+\infty}\mbox{tr}(\Lambda_{k})=0$ from Lemma \ref{bkl}(ii).

To prove the strong consistency,  let $\mathcal{G}_{k}=\sigma\lbrace V_{l}, 0\leq l\leq k-1\rbrace$. Then,
\dref{1w} and Corollary \ref{lsl} yield
\begin{eqnarray*}\label{kjy}
E[\widetilde\Theta_{k+1}^{\tau}\widetilde\Theta_{k+1}|\mathcal{G}_{k}]\leq \widetilde\Theta_{k}^{\tau}\widetilde\Theta_{k}-s\sum_{j=1}^{n}(1-(P_{k+1,j}P_{k,j}^{-1})^2)\|\widetilde\Theta_{k}\|^2+nM\sum_{i=1}^{n}P_{k+1,i}^2\phi_{k,i}^2.
\end{eqnarray*}
Since $\widetilde\Theta_{k}^{\tau}\widetilde\Theta_{k}\in\mathcal{G}_{k}$, by  \dref{wsj} and
 Lemma \ref{ty}, $
\widetilde\Theta_{k+1}^{\tau}\widetilde\Theta_{k+1}\rightarrow0
$ as $k\rightarrow+\infty$   almost surely
with the convergence rate
\begin{eqnarray}\label{sumPThe}
\sum_{k=0}^{+\infty}\sum_{i=1}^{n}(1-(P_{k+1,i}P_{k,i}^{-1})^2)\|\widetilde\Theta_{k}\|^2<+\infty,\quad\mbox{a.s.}.
\end{eqnarray}

Now, we prove the necessity under
$$\lim_{k\rightarrow+\infty}\sum_{i=1}^{n}P_{k+1,i}^{-1}<+\infty.$$
In this case,
\begin{eqnarray}\label{pk+lphi}
\sum_{k=0}^{+\infty}\sum_{i=1}^{n}P_{k+1,i}\phi_{k,i}^2<+\infty.
\end{eqnarray}
Denote $\Pi_{k}\triangleq\prod_{i=k}^{0}\mathcal{A}(I_n-F_i),$
we first prove  $\lim_{k\rightarrow+\infty}\Pi_{k}$ exists. 
In fact,  since $\mathcal{A}$ is an  irreducible and aperiodic doubly stochastic matrix, we have $\lim_{k\rightarrow}\mathcal{A}^{k}=\frac{1}{n}\cdot\textbf{1}\textbf{1}^{\tau}$. Then, by \dref{pk+lphi}, given
any $\varepsilon>0$, 
there is a $k_1>0$ such that
 \begin{eqnarray}\label{APphi}
\left\{
\begin{array}{l}
\sum_{k=k_1}^{+\infty}\sum_{i=1}^{n}P_{k+1,i}\phi_{k,i}^2<\frac{\varepsilon}{3}\\
\|\mathcal{A}^{k}-\mathcal{A}^{l}\|_1<\frac{\varepsilon}{3},\quad \forall k,l>k_1
\end{array},
\right.
\end{eqnarray}
here $\|X\|_1\triangleq \max_{1\leq j\leq r}\sum_{i=1}^{p}X[i,j]$ for any $X\in \mathbb{R}^{p\times r}$, $p,r\geq 1$.

Therefore, for every $k>2k_1$,
\begin{eqnarray}\label{Pijj1}
\|\Pi_{k}-\mathcal{A}^{k-k_1}\Pi_{k_1}\|_1 &=&\left\|\sum_{j=k_1}^{k-1}\mathcal{A}^{k-j-1}(\Pi_{j+1}-\mathcal{A}\Pi_{j})\right\|_1=\left\|\sum_{j=k_1}^{k-1}\mathcal{A}^{k-j}F_{j+1}\Pi_{j}\right\|_1
\nonumber\\ &\leq &\sum_{j=k_1}^{k-1}\left\|\mathcal{A}^{k-j}F_{j+1}\Pi_{j}\right\|_1\leq\sum_{j=k_1}^{k-1}\sum_{i=1}^{n}P_{j+2,i}\phi_{j+1,i}^2<\frac{\varepsilon}{3}.
\end{eqnarray}
Combining \dref{APphi} and \dref{Pijj1} infers that  for all $k,l>2k_1$,
\begin{eqnarray}\label{pkl}
\|\Pi_{k}-\Pi_{l}\|_1 &\leq &\|\Pi_{k}-\mathcal{A}^{k-k_1}\Pi_{k_1}\|_1+\|\Pi_{l}-\mathcal{A}^{l-k_1}\Pi_{k_1}\|_1+\|(\mathcal{A}^{l-k_1}-\mathcal{A}^{k-k_1})\Pi_{k_1}\|_1\nonumber\\
&<&\frac{\varepsilon}{3}+\frac{\varepsilon}{3}+\frac{\varepsilon}{3}=\varepsilon,
\end{eqnarray}
which means $\lim_{k\rightarrow+\infty}\Pi_{k}$ exists.

Now, denote $\Pi\triangleq\lim_{k\rightarrow+\infty}\Pi_{k}$. Observe that
 \begin{eqnarray*}
\left\{
\begin{array}{l}
\Pi_{k+1}=\mathcal{A}(I_n-F_{k+1})\Pi_{k}\\
\lim_{k\rightarrow+\infty}\mathcal{A}(I_n-F_k)=\mathcal{A}
\end{array},
\right.
\end{eqnarray*}
then $\Pi=\mathcal{A}\Pi$. Consequently,  $\Pi=\textbf{1}\cdot (\mu_1,\ldots,\mu_n)$ for some $\mu_i\geq 0,  i=1,\ldots,n$.
We now prove  $\mu_i>0$ for all  $ i=1,\ldots,n$. First,   \dref{pk+lphi} infers that there is a $k_2>k_1$ such that $\sum_{k=k_2}^{+\infty}\sum_{i=1}^{n}P_{k+1,i}\phi_{k,i}^2<1$. Furthermore,
 $1-P_{k+1,i}\phi_{k,i}^2>0$ for all $k\geq 0$, $i=1,\ldots,n$ and $\lim_{k\rightarrow}\mathcal{A}^{k}=\frac{1}{n}\cdot\textbf{1}\textbf{1}^{\tau}$,
 we then conclude that as long as $k_2$ is sufficiently large,
\begin{eqnarray}\label{Pik2}
\min_{i,j}\Pi_{k_2}[i,j]>0.
\end{eqnarray}
Further, since $\mathcal{A}$ is a doubly stochastic matrix,  for all $k\geq0$,
\begin{eqnarray}
\min_{i,j}\Pi_{k+1}[i,j]
\geq \left(1-\max_{j}P_{k+1,j}\phi_{k,j}^2\right)\min_{i,j}\Pi_{k}[i,j].\nonumber
\end{eqnarray}
As a result, by \dref{APphi} and \dref{Pik2},
\begin{eqnarray}
&&\min\lbrace \mu_1,\ldots,\mu_n\rbrace=\liminf_{k\rightarrow+\infty}\min_{i,j}\Pi_{k+1}[i,j]\geq\min_{i,j}\Pi_{k_2}[i,j]\prod_{k=k_2}^{+\infty}\left(1-\sum_{i=1}^{n}P_{k+1,i}\phi_{k,i}^2\right)>0.\nonumber
\end{eqnarray}
So, in view of \dref{jj},
\begin{eqnarray}\label{trthe0}
\liminf_{k\rightarrow+\infty}\mbox{tr}(\Lambda_{k+1})
&\geq &\liminf_{k\rightarrow+\infty}\mbox{tr}(\Pi_k\widetilde\Theta_{0}\widetilde\Theta_{0}^{\tau}\Pi_{k}^{\tau})=n\left(\sum_{j=1}^{n}\mu_j(\theta_{0,j}-\theta)\right)^2,
\end{eqnarray}
which infers $\liminf\limits_{k\rightarrow+\infty}E\widetilde\Theta_{k}^{\tau}\widetilde\Theta_{k}>0$  if  $\sum_{j=1}^{n}\mu_j(\theta_{0,j}-\theta)\not=0$.

The last part is addressed to proving $\widetilde\Theta_{k}\stackrel{p}{\nrightarrow}0$.
By  \dref{wss} and \dref{wsj}, Lemma \ref{bkl}(i) shows that $\lim_{k\rightarrow+\infty}E\|\widetilde\Theta_{k}\|^2$ exists. Denote
 \begin{eqnarray*}
\left\{
\begin{array}{l}
Q\triangleq\sup_{k\geq 0}E\|\widetilde\Theta_{k}\|^2\\
\Pi(k,i)\triangleq\prod_{j=k}^{k-i+1}\mathcal{A}(I_{n}-F_{k})
\end{array}.
\right.
\end{eqnarray*}
By \dref{wsj},  for any fixed $\varepsilon>0$, there is a $k_3>0$ such that
\begin{eqnarray}
\sum_{j=k_3}^{+\infty}\sum_{i=1}^{n}P_{j+1,i}^2\phi_{j,i}^2<\frac{\varepsilon}{4M},\nonumber
\end{eqnarray}
In addition, similar to \dref{pkl},  there is a $k_4>k_3$ such that for any $k,l>k_4$,
\begin{eqnarray}
\|\Pi(k,k-k_3-1)-\Pi(l,l-k_3-1)\|_2<\frac{\varepsilon}{2Q}.\nonumber
\end{eqnarray}

So, as long as  $k,l>k_4$,
\begin{eqnarray*}
\widetilde\Theta_{k+1}-\widetilde\Theta_{l+1}
&=&\Pi(k,k-k_3+1)\widetilde\Theta_{k_3}+\sum_{j=k_3}^{k}\Pi(k,k-j)\mathcal{A}L_{j}V_j\nonumber\\
&&-\Pi(l,l-k_3+1)\widetilde\Theta_{k_3}+\sum_{j=k_3}^{l}\Pi(l,l-j)\mathcal{A}L_{j}V_j,
\end{eqnarray*}
which infers
\begin{eqnarray}
&&E\|\widetilde\Theta_{k+1}-\widetilde\Theta_{l+1}\|^2\nonumber\\
&\leq & E\|(\Pi(k,k-k_3+1)-\Pi(l,l-k_3+1))\widetilde\Theta_{k_3}\|^2+2\sum_{j=k_3}^{\max\lbrace k,l\rbrace}E\|L_{j}V_j\|^2\nonumber\\
&<&\frac{\varepsilon}{2Q}\cdot Q+2M\sum_{j=k_3}^{+\infty}\sum_{i=1}^{n}P_{j+1,i}^2\phi_{j,i}^2<\varepsilon.\nonumber
\end{eqnarray}
This means $\lbrace \widetilde\Theta_{k}\rbrace_{k\geq 0}$ is a Cauchy sequence in $L^2(dP)$, and hence there exists a random vector $Z\in L^2(dP)$ such that $\lim_{k\rightarrow+\infty}E\|\widetilde\Theta_{k}-Z\|^2=0$. So, $\widetilde\Theta_{k}\stackrel{p}{\longrightarrow} Z$. 
Note that $Z\neq 0$ due to  $\lim_{k\rightarrow+\infty}E\|\widetilde\Theta_{k}\|^2\neq0$. 
\end{proof}

\section{Proofs of Theorems  \ref{ls2'}--\ref{ls2}}\label{proofls2}

Since a deterministic   parameter can be viewed as a   random variable having a degenerate Gaussian distribution with zero variance,  it suffices to prove  Remark \ref{randomthe} by  assuming that $\theta$ in   Theorems \ref{ls2'}--\ref{ls2} is random. In addition, let $\theta$ in Theorem \ref{ls2} be  Gaussian distributed.

We first prove a technical lemma.
Fix a $j^{*}\in\lbrace 1,\ldots,n\rbrace$ and let $d$ be the smallest integer that $\mathcal{A}^{d+1}[j^{*},j^{*}]>0$.   define a sequence of vectors in $\mathbb{R}^{mn}$:
 \begin{eqnarray*}
\left\{
\begin{array}{l}
\mathcal{P}_0\triangleq\left\lbrace C: C[m(j^{*}-1)+1,1]=0\right\rbrace\\
\mathcal{P}_l\triangleq\left\lbrace C: \sum_{k=1}^{n}b_{l,k}C[m(k-1)+1,1]=0\right\rbrace, l\in[1,d]
\end{array},
\right.
\end{eqnarray*}
where for $l\geq 1$ and $k=1,\ldots,n$,
$$b_{l,k}=\sum_{i\not=j^{*}}\sum_{i_1,\ldots,i_{l-1}}a_{j^{*}i}a_{ii_1}\ldots a_{i_{l-1}k}.$$

\begin{lemma}\label{two1}
Given $k\geq 1$, let $f=(f_1,\ldots,f_{mn})^{\tau}: \mathbb{R}^{k}\rightarrow \mathbb{R}^{mn}$ be a map that each $f_i(z)$ is a polynomial of $z\in \mathbb{R}^k, 1\leq i\leq mn$.   If
$f(\mathbb{R}^{k})\not\subset \mathcal{P}_l$ for some $l\in[0,d]$,
then for any nonempty  open set $U\subset\mathbb{R}^{k}$, there is a $z\in U$ such that $f(z)\not\in \mathcal{P}_l$.
\end{lemma}
\begin{proof}
Since $\sum_{k=1}^{n}b_{l,k}f_{m(k-1)+1}(z)$ is a polynomial, if for some nonempty  open set $U\subset\mathbb{R}^{k}$,
$$
\sum_{k=1}^{n}b_{l,k}f_{m(k-1)+1}(z)\equiv 0 \quad \mbox{for all}\,\,z\in U
$$
then the polynomial  must be identically zero on  $\mathbb{R}^{k}$. This contradicts to
$f(\mathbb{R}^{k})\not\subset \mathcal{P}_l, l\in[0,d]$.
\end{proof}

\begin{lemma}\label{four}
Let $C\in \mathbb{R}^{mn}$ be a vector and   $B_i\in\mathbb{R}^{m\times m}$, $i=1,\ldots,n$ be a sequence of positive definite matrices. Define a map $Q_0: \mathbb{R}^{mn}\rightarrow \mathbb{R}^{mn\times mn}$ by
\begin{eqnarray}
Q_0(z)\triangleq \mbox{diag}\lbrace(B_i^{-1}+v_iv_i^{\tau})^{-1}B_i^{-1},i=1,\ldots,n\rbrace,\nonumber
\end{eqnarray}
where $z=\mbox{col}\lbrace v_1,\ldots,v_n\rbrace$ and $v_{i}\in\mathbb{R}^{m}$, $1\leq i\leq n$. For each $l\in [0,d]$,\\
(i) if $C\not\in \mathcal{P}_{l+1}$, then for any nonempty open set $U\in\mathbb{R}^{mn}$,  there is a $z\in U$ such that
\begin{eqnarray}
(\mathcal{A}\otimes I_{m})(Q_{0}(z)C)\not\in\mathcal{P}_{l};\nonumber
\end{eqnarray}
(ii) if $C\not\in \mathcal{P}_{0}$, then for any nonempty open set $U\in\mathbb{R}^{mn}$,  there is a $z\in U$ such that
\begin{eqnarray}
(\mathcal{A}\otimes I_{m})(Q_{0}(z)C)\notin\mathcal{P}_{d}.\nonumber
\end{eqnarray}
\end{lemma}
\begin{proof}
(i) Let $D(z)\triangleq(\mathcal{A}\otimes I_{m})Q_0(z)C$, then
\begin{eqnarray}
D[m(i-1)+1,1](z)&=&-\sum_{k=1}^{n}a_{ik}\frac{v_k^{\tau}C_{1}^{(k)}B_kv_k}{1+v_k^{\tau}B_kv_k}+\sum_{k=1}^{n}a_{ik}C[m(k-1)+1,1],\nonumber
\end{eqnarray}
where $C_{1}^{(k)}\in\mathbb{R}^{m\times m}$ satisfies
\begin{eqnarray*}
C_1^{(k)}=\begin{bmatrix}
C[m(k-1)+1,1] & 0& \ldots & 0\\
\vdots & \vdots & \ddots & \vdots\\
C[mk,1]& 0 & \ldots &0\\
\end{bmatrix},\quad k=1,\ldots,n.
\end{eqnarray*}
Becasue each component of  $ D(z)\prod_{k=1}^n (1+v_k^{\tau}B_kv_k)$ is a polynomial and $\prod_{k=1}^n (1+v_k^{\tau}B_kv_k)>0$, in view of Lemma \ref{two1}, it is sufficient to prove
$$     D(\mathbb{R}^{mn})\prod_{k=1}^n (1+v_k^{\tau}B_kv_k)\not\subset\mathcal{P}_l.$$

Suppose $(\mathcal{A}\otimes I_{m})(Q_{0}(\mathbb{R}^{mn})C)\subset\mathcal{P}_{l}$ and let $v_j=(x_{1j},\ldots,x_{mj})^{\tau}$, $j\in [1,n]$.   If  $l\geq 1$,  the constant term of
\begin{eqnarray}
\prod_{k=1}^{n}(1+v_k^{\tau}B_kv_k)\sum_{k=1}^{n}b_{l,k}D[m(k-1)+1,1]\nonumber
\end{eqnarray}
is
\begin{eqnarray}
\sum_{k=1}^{n}\sum_{i=1}^{n}b_{l,i}a_{i,k}C[m(k-1)+1,1]=\sum_{k=1}^{n}b_{l+1,k}C[m(k-1)+1,1]=0,\nonumber
\end{eqnarray}
which implies $C\in \mathcal{P}_{l+1}$. It contradicts to $C\not\in \mathcal{P}_{l+1}$.
If $l=0$,  the  coefficient of $x_{1j^{*}}^2$ of
\begin{eqnarray}
\prod_{k=1}^{n}(1+v_k^{\tau}B_kv_k)D[m(j^{*}-1)+1,1]\nonumber
\end{eqnarray}
is
\begin{eqnarray}
B_{j^{*}}[1,1]\sum_{k\not=j^{*}}a_{j^{*}k}C[m(k-1)+1,1]=0,\nonumber
\end{eqnarray}
which implies $C\in \mathcal{P}_1$ since $B_j$ is positive definite. Hence,  it leads to a contradiction again.\\
(ii) If $(\mathcal{A}\otimes I_{m})(Q_{0}(\mathbb{R}^{mn})C)\subset\mathcal{P}_{d}$, then the coefficient of $x_{1j^{*}}^2$  and the constant term of
\begin{eqnarray}\label{1vb}
\prod_{k=1}^{n}(1+v_k^{\tau}B_kv_k)\sum_{k=1}^{n}b_{d,k}D[m(k-1)+1,1]\nonumber
\end{eqnarray}
are
\begin{eqnarray}\label{2vb}
B_{j^{*}}[1,1]\sum_{k\not=j^{*}}\sum_{i=1}^{n}b_{d,i}a_{i,k}C[m(k-1)+1,1]=0,\nonumber
\end{eqnarray}
and
\begin{eqnarray}
\sum_{k=1}^{n}\sum_{i=1}^{n}b_{d,i}a_{i,k}C[m(k-1)+1,1]=0,\nonumber
\end{eqnarray}
respectively. As a result,
\begin{eqnarray}
\sum_{i=1}^{n}b_{d,i}a_{i,j^{*}}C[m(j^{*}-1)+1,1]
=\sum_{i\not=j^{*}}\sum_{i_1,\ldots,i_{d}}a_{j^{*}i}a_{ii_1}\ldots a_{i_{d}j^{*}}C[m(j^{*}-1)+1,1]=0.\nonumber
\end{eqnarray}
So, by $C\notin \mathcal{P}_0$,  $$\sum_{i\not=j^{*}}\sum_{i_1,\ldots,i_{d}}a_{j^{*}i}a_{ii_1}\ldots a_{i_{d}j^{*}}=0.$$
Hence
\begin{eqnarray}
\mathcal{A}^{d+1}[j^{*},j^{*}]&=&\sum_{i=1}^{n}\sum_{i_1,\ldots,i_{d}}a_{j^{*}i}a_{ii_1}\ldots a_{i_{d}j^{*}}=a_{j^{*}j^{*}}\left(\sum_{i_1,\ldots,i_{d}}a_{j^{*}i_1}\ldots a_{i_{d}j^{*}}\right),\nonumber
\end{eqnarray}
 which together with $\mathcal{A}^{d+1}[j,j]>0$ implies
\begin{eqnarray}
\mathcal{A}^{d}[j^{*},j^{*}]=\sum_{i_1,\ldots,i_{d}}a_{j^{*}i_1}\ldots a_{i_{d}j^{*}}>0.\nonumber
\end{eqnarray}
This contradicts to the definition of $d$.
\end{proof}
Now, letting
\begin{eqnarray}
z=\mbox{col}\lbrace\underbrace{\textbf{0},\ldots, \textbf{0}}_{j^{*}-1},v,\underbrace{\textbf{0},\ldots, \textbf{0}}_{n-j^{*}}\rbrace,\nonumber
\end{eqnarray}
in Lemma \ref{four} shows
\begin{corollary}\label{four1}
Let $C\in \mathbb{R}^{mn}$ and $B\in\mathbb{R}^{m\times m}$ be a vector and a positive definite matrix.  Denote $Q_0^{*}: \mathbb{R}^{m}\rightarrow \mathbb{R}^{mn\times mn}$ by
\begin{eqnarray}
Q_0^{*}(v)\triangleq\mbox{diag}\lbrace\underbrace{I_m,\ldots, I_m}_{j^{*}-1},(B^{-1}+vv^{\tau})^{-1}B^{-1},\underbrace{I_m,\ldots, I_m}_{n-j^{*}}\rbrace,\nonumber
\end{eqnarray}
\\
(i) If $C\not\in \mathcal{P}_{l+1}$, then for any nonempty open set $U\in\mathbb{R}^{m}$,  there is $z\in U$ such that
\begin{eqnarray}
(\mathcal{A}\otimes I_{m})(Q_{0}^{*}(z)C)\not\in\mathcal{P}_{l}.\nonumber
\end{eqnarray}
(ii) If $C\not\in \mathcal{P}_{0}$, then for any nonempty open set $U\in\mathbb{R}^{m}$,  there is $z\in U$ such that
\begin{eqnarray}
(\mathcal{A}\otimes I_{m})(Q_{0}^{*}(z)C)\not\subset\mathcal{P}_{d}.\nonumber
\end{eqnarray}
\end{corollary}

The next lemma with the proof given in Appendix \ref{AppA}   is the main reason for  the failure of the diffusion RLS in Remark \ref{randomthe}.
We introduce some necessary  notations.
For  $C\in \mathbb{R}^{mn}$ and $B\in\mathbb{R}^{m\times m}$  defined in Corollary \ref{four1}, denote  maps $Q_1: \mathbb{R}^{m}\rightarrow \mathbb{R}^{mn\times mn}$, $Q_2: \mathbb{R}^{m}\times \mathbb{R}^{m}\rightarrow \mathbb{R}^{mn\times mn}$ and $Q_3: \mathbb{R}^{m}\times \mathbb{R}^{m}\rightarrow \mathbb{R}^{mn\times mn}$ by
\begin{eqnarray}
\left\{
\begin{array}{l}
Q_1(v_1)\triangleq \mbox{diag}\lbrace\underbrace{I_m,\ldots, I_m}_{j^{*}-1},B_1B^{-1},\underbrace{I_m,\ldots, I_m}_{n-j^{*}}\rbrace\nonumber\\
Q_2(v_1,v_2)
\triangleq \mbox{diag}\lbrace\underbrace{I_m,\ldots, I_m}_{j^{*}-1},B_2B_1^{-1},\underbrace{I_m,\ldots, I_m}_{n-j^{*}}\rbrace\nonumber\\
Q_3(v_1,v_2)\triangleq (\mathcal{A}\otimes I_{m}) Q_2(v_1,v_2)(\mathcal{A}\otimes I_{m})Q_1(v_1)C\nonumber
\end{array},
\right.
\end{eqnarray}
where $B_1\triangleq(B^{-1}+v_1v_1^{\tau})^{-1}$, $B_2\triangleq(B^{-1}+v_1v_1^{\tau}+v_2v_2^{\tau})^{-1}$ and $v_1,v_2\in \mathbb{R}^{m}$.

\begin{lemma}\label{12}
Let $a_{lj^{*}}>0$ for some $l\in [1,n]$, where $j^{*}$ is the fixed index defined before. If $C\not\in\mathcal{P}_1$, then for any $L>0$, there exist some $v_1,v_2\in \mathbb{R}^{m}$ such that
\begin{eqnarray*}
\left\{
\begin{array}{l}
(\mathcal{A}\otimes I_{m})(Q_1(v_1))C\not\in \mathcal{P}_0\\
Q_3(v_1,v_2)\not\in \mathcal{P}_d\\
|Q_3(v_1,v_2)[m(l-1)+1,1]|>L
\end{array}.
\right.
\end{eqnarray*}
\end{lemma}
\begin{lemma}\label{13}
Let $C\in \mathbb{R}^{mn}$ and $\{B_i\in\mathbb{R}^{m\times m}\}$ be defined in Lemma \ref{four}. For any $K>0$, if $C\not\in\mathcal{P}_{d}$, then there exists some $z_j=\mbox{col}\lbrace v_{j,1},\ldots,v_{j,n}\rbrace\in\mathbb{R}^{mn}$, $j\in[1,m]$, such that
\begin{eqnarray*}
\left\{
\begin{array}{l}
\inf_{i\in[1,n]}\lambda_{min}\left(B^{-1}_i+\sum_{j=1}^{m}v_{j,i}v_{j,i}^{\tau}\right)>K\\
\prod_{k=j}^{1}(\mathcal{A}\otimes I_m)G_k(z_1,\ldots,z_k)C\not\in\mathcal{P}_{d-j},\quad j\in[1,m]
\end{array},
\right.
\end{eqnarray*}
 where $\mathcal{P}_{-l}\triangleq\mathcal{P}_{d-l+1}$, $l\geq 1$ and
 \begin{eqnarray*}
\left\{
\begin{array}{l}
G_j(z_1,\ldots,z_j)\triangleq \mbox{diag} \left\lbrace B_{1,j}B_{1,j-1}^{-1},\ldots,B_{n,j}B_{n,j-1}^{-1}\right\rbrace\\
B_{i,j}\triangleq (B_i^{-1}+\sum_{k=1}^{j}v_{k,i}v_{k,i}^{\tau})^{-1}, \quad 1\leq i\leq n
\end{array}
\right.
\end{eqnarray*}
for  $1\leq j\leq m$ and  $B_{i,0}\triangleq B_{i}$, $1\leq i\leq n$.
\end{lemma}

\begin{proof}
Let $e_{j}$ denote the $j$th column of the identity matrix $I_m$, $j\in[1,m]$ and
$$z_j^{*}=\mbox{col}\lbrace v_{j,1}^{*},\ldots,v_{j,n}^{*}\rbrace=\mbox{col}\lbrace \sqrt{nK} \cdot e_{j},\ldots,\sqrt{nK}\cdot e_{j}\rbrace.$$
Then,  for $i\in[1,n]$,
\begin{eqnarray}
\lambda_{min}\left(B^{-1}_i+\sum_{j=1}^{m}v_{j,i}^{*}(v_{j,i}^{*})^{\tau}\right)&\geq &\lambda_{min}\left(\sum_{j=1}^{m}v_{j,i}^{*}(v_{j,i}^{*})^{\tau}\right)=nK>K.\nonumber
\end{eqnarray}
Since
$$\lambda(z_1, z_{2},\ldots, z_{m})\triangleq\inf_{i\in[1,n]}\lambda_{min}\left(B^{-1}_i+\sum_{j=1}^{m}v_{j,i}v_{j,i}^{\tau}\right)$$
 is continuous in $z_1,\ldots,z_m$, there exists a neighbourhood $U_1$ of $z_1^{*}$ such that  $\lambda(s, z_{2}^{*},\ldots, z_{m}^{*})>K$ for all $s\in U_1$. By Lemma \ref{four}, there is a $z_1\in U_1$ such that $(\mathcal{A}\otimes I_m)G_1(z_1)C\not \in \mathcal{P}_{d-1}$.  An analogous argument shows that we can select a series of  $z_1,\ldots,z_m$   satisfying
 \begin{eqnarray*}
\left\{
\begin{array}{l}
\inf_{i\in[1,n]}\lambda_{min}\left(B^{-1}_i+\sum_{j=1}^{m}v_{j,i}v_{j,i}^{\tau}\right)>K\\
\prod_{k=j}^{1}(\mathcal{A}\otimes I_m)G_k(z_1,\ldots,z_k)C\not\in\mathcal{P}_{d-j}
\end{array},\,\,\, j\in[1,m],
\right.
\end{eqnarray*}
which  is exactly the result as desired.
\end{proof}

\begin{lemma}\label{gj}
Let $E\widetilde\Theta_{0}[m(j^{*}-1)+1,1]\not=0$ and   $a_{lj^{*}}>0$ for some $l\in[1,n]$. Then,  under Assumption A1', there is a sequence of deterministic matrices $\lbrace \Phi_{i}\rbrace_{i=0}^{+\infty}$ such that $$\lim_{t\rightarrow+\infty}\lambda_{min}\left(\sum_{i=0}^{t}\Phi_i\Phi_i^{\tau}\right)=+\infty$$
and for  $R_t\triangleq\prod_{i=t}^{0}(\mathcal{A}\otimes I_{m})(I_{mn}-F_i)E\widetilde\Theta_{0}$,
\begin{eqnarray}
\limsup_{t\rightarrow+\infty}\dfrac{|R_t[m(l-1)+1,1]|}{16(t+1)^4}> 1.\nonumber
\end{eqnarray}
\end{lemma}

\begin{proof}
It suffices to  construct a series of deterministic $\lbrace \Phi_{i}\rbrace_{i=0}^{+\infty}$   such that for any $ k\geq 0$, $s\in[0,d]$ 
 and $t_k=k(m+3)(d+1)$,
 \begin{eqnarray}\label{dd}
\left\{
\begin{array}{l}
R_{t_k+j}\not\in\mathcal{P}_{d-j}\\
\lambda_{min}\left(\sum_{i=0}^{t_{k}+m}\Phi_{i}\Phi_{i}^{\tau}\right)>t_{k}+m\\
|R_{t_{k+1}}[m(l-1)+1,1]|>20(t_{k+1}+1)^4
\end{array}.
\right.
\end{eqnarray}

 First,  since $E\widetilde\Theta_{0}\not\in\mathcal{P}_0$, by Lemma \ref{four}, there is a
$\Phi_0$ such that  $R_0\not\in\mathcal{P}_d$.
Let $k=0$. In view of Lemma \ref{13}, we can find some
$\Phi_{t_k+j},j=1\ldots,m$,
 such that for all $j\in[1,m]$,
 \begin{eqnarray}\label{ddk=0}
\left\{
\begin{array}{l}
\lambda_{min}\left(\sum_{i=0}^{t_k+m}\Phi_i\Phi_i^{\tau}\right)>t_k+m\\
R_{t_k+j}=\prod_{i=j}^{1}(\mathcal{A}\otimes I_m)(I_{mn}-F_{t_k+i})R_{t_k}\not\in \mathcal{P}_{d-j}
\end{array}.
\right.
\end{eqnarray}
Moreover, by Lemma \ref{four}, there are some  $\Phi_{t_k+j},j=m+1,\ldots,(m+3)(d+1)-2$   
such that for all $j\in[m+1,(m+3)(d+1)-2]$,
\begin{eqnarray*}
R_{t_k+j}=\prod_{i=j}^{m+1}(\mathcal{A}\otimes I_m)(I_{mn}-F_{t_k+i})R_{t_k+m}\not\in \mathcal{P}_{d-j}.
\end{eqnarray*}
Finally,    by  noting that
$R_{t_k+(m+3)(d+1)-2}\not\in\mathcal{P}_{1}$,
Lemma \ref{12} indicates  that for some
 $$\Phi_{t_{k+1}-i}= \mbox{diag}\lbrace\underbrace{\textbf{0},\ldots, \textbf{0}}_{j^{*}-1},v_{i,j^*},\underbrace{\textbf{0},\ldots, \textbf{0}}_{n-j^{*}}\rbrace,\quad i=0,1,   $$
one has
  \begin{eqnarray}\label{ddk+1}
\left\{
\begin{array}{l}
R_{t_{k+1}-1}\not\in\mathcal{P}_{0},\quad R_{t_{k+1}}\not\in\mathcal{P}_{d}\\
|R_{t_{k+1}}[m(l-1)+1,1]|>20(t_{k+1}+1)^4
\end{array}.
\right.
\end{eqnarray}
So,  we obtain a series of $\{\Phi_j, j=0,\ldots,t_1\}$ fulfilling \dref{dd}.
By repeating \dref{ddk=0} to \dref{ddk+1} for all $k\geq 1$, \dref{dd} is proved immediately  based on the  mathematical induction.
\end{proof}

\begin{proof}[Proof of Theorem \ref{ls2'}]
Considering $\|E\widetilde{\Theta}_0\|\neq 0$ and    Assumption A1',  we suppose,   without loss of generality, there are some $j^{*}, l \in\lbrace 1,\ldots,n\rbrace$  such that      $E\widetilde\Theta_{0}[m(j^{*}-1)+1,1]\not=0$ and $a_{lj^{*}}>0.$
Let  $\lbrace\Phi_{k}\rbrace_{k=0}^{+\infty}$ be the  deterministic   sequence constructed in Lemma \ref{gj}.  Then, by virtue of   Assumption A2 and  \dref{1w},
\begin{eqnarray}\label{ETheR}
E(\widetilde\Theta_{k}[m(l-1)+1,1])=R_k[m(l-1)+1,1],
\end{eqnarray}
and hence
\begin{eqnarray*}
\sup_{k\geq 0}E\|\widetilde\Theta_{k}\|^2&\geq &\sup_{k\geq 0}\|E \widetilde\Theta_{k}\|^2\geq \sup_{k\geq 0}(E(\widetilde\Theta_{k}[m(l-1)+1,1]))^2\nonumber\\
&=&\sup_{k\geq 0}(R_k[m(l-1)+1,1])^2=+\infty,
\end{eqnarray*}
where $R_k$ is define in Lemma \ref{gj}.
\end{proof}

\begin{proof}[Proof of Theorem \ref{ls2}]
Let  $\lbrace\Phi_{k}\rbrace_{k=0}^{+\infty}$ be defined   in the proof of Theorem \ref{ls2'}.  Since  $\theta$ is Gaussian distributed,  $\widetilde\Theta_{k}[m(l-1)+1,1]$ possesses a normal distribution  by  Assumption A2''.  Note that
  for any random variable $\xi\sim N(E\xi, \sigma^2)$ and $k\geq 1$,
\begin{eqnarray}\label{st}
P\left(16(k+1)^{3}|\xi|< |E\xi|\right)
&=& I_{\lbrace\sigma\not=0\rbrace}\cdot\frac{1}{\sqrt{2\pi}}\int_{-\frac{1}{16(k+1)^{3}}\frac{|E\xi|}{|\sigma|}-\frac{E\xi}{\sigma}}^{\frac{1}{16(k+1)^{3}}\frac{|E\xi|}{|\sigma|}-\frac{E\xi}{\sigma}}e^{-\frac{x^2}{2}}dx\nonumber\\
&\leq & I_{\left\lbrace \frac{|E\xi|}{|\sigma|}<8k\right\rbrace}\cdot\frac{1}{\sqrt{2\pi}}\frac{1}{(k+1)^2}+I_{\left\lbrace \frac{|E\xi|}{|\sigma|}\geq 8k\right\rbrace}\cdot \frac{1}{\sqrt{2\pi}}\int_{4k}^{+\infty}e^{-\frac{x^2}{2}}dx\nonumber\\
&\leq & I_{\left\lbrace \frac{|E\xi|}{|\sigma|}<8k\right\rbrace}\cdot\frac{1}{\sqrt{2\pi}}\frac{1}{(k+1)^2}+I_{\left\lbrace \frac{|E\xi|}{|\sigma|}\geq 8k\right\rbrace}\cdot \frac{1}{\sqrt{2\pi}}\int_{4k}^{+\infty}\frac{2}{x^3}dx\nonumber\\
&\leq &\frac{1}{\sqrt{2\pi}}\frac{1}{k^2},
\end{eqnarray}
and
\begin{eqnarray}\label{gl}
P(|\xi|\leq \varepsilon)&=& I_{\lbrace\sigma\not=0\rbrace}\cdot\frac{1}{\sqrt{2\pi}}\int_{-\frac{\varepsilon}{|\sigma|}-\frac{E\xi}{\sigma}}^{\frac{\varepsilon}{|\sigma|}-\frac{E\xi}{\sigma}}e^{-\frac{x^2}{2}}dx+I_{\lbrace\sigma=0, |E\xi|\leq \varepsilon\rbrace}\nonumber\\
&\leq & 2\cdot I_{\lbrace |E\xi|\leq \varepsilon\rbrace}+I_{\lbrace\sigma>\varepsilon\rbrace}\cdot \frac{2}{\sqrt{2\pi}}+I_{\lbrace\sigma\leq \varepsilon, |E\xi|> \varepsilon\rbrace}\cdot\frac{1}{\sqrt{2\pi}}\int_{\frac{1}{\varepsilon}(|E\xi|-\varepsilon)}^{+\infty}e^{-\frac{x^2}{2}}dx.\nonumber
\end{eqnarray}
Define  $$D_0\triangleq\bigcap_{k=1}^{+\infty}\left\lbrace |\widetilde\Theta_{k}[m[l-1]+1,1]|\geq\dfrac{|R_k[m(l-1)+1,1]|}{16(k+1)^{3}}\right\rbrace,$$
then \dref{st} infers
\begin{eqnarray}
P(D_0)
&\geq & 1-\sum_{k=1}^{+\infty}P(\lbrace 16(k+1)^{3}|\widetilde\Theta_{k}[m(l-1)+1,1]|< |R_k[m(l-1)+1,1]|\rbrace)\nonumber\\
&\geq & 1-\frac{1}{\sqrt{2\pi}}\sum_{k=1}^{+\infty}\frac{1}{k^2}>0.\nonumber
\end{eqnarray}
According to Lemma \ref{gj} and \dref{ETheR},
\begin{eqnarray*}
\|\widetilde\Theta_{k}\|&\geq& |\widetilde\Theta_{k}[m(l-1)+1,1]|\geq  \frac{1}{16(k+1)^3}R_{k}[m(l-1)+1,1]\nonumber\\
&>&k+1,\quad  \mbox{i.o. on}~D_0.
\end{eqnarray*}
So,   \dref{ls2as} holds.  Moreover, by    Lemma \ref{gj} and \dref{ETheR} again,
$$\limsup_{k\rightarrow+\infty}|E\widetilde\Theta_{k}[m(l-1)+1,1]|=+\infty,$$
which together with \dref{gl} yields
\begin{eqnarray}
\liminf_{k\rightarrow+\infty}P(|\widetilde\Theta_{k}[m(l-1)+1,1]|\leq \varepsilon)\leq \frac{2}{\sqrt{2\pi}},\nonumber
\end{eqnarray}
 and hence
\begin{eqnarray*}
\limsup_{k\rightarrow+\infty}P(\|\widetilde\Theta_{k}\|>\varepsilon)\geq \limsup_{k\rightarrow+\infty}P(|\widetilde\Theta_{k}[m(l-1)+1,1]|
>\varepsilon)
\geq 1-\frac{2}{\sqrt{2\pi}}.
\end{eqnarray*}
The  proof is completed.
\end{proof}



\section{Concluding Remarks}\label{concluding}

We have established   the necessary and sufficient condition that ensures  the   diffusion RLS  converging to  true scalar  parameters.  This  condition shows that
cooperations among nodes  through diffusion networks indeed could help estimation, as long as the
 parameters to be identified are scalar. But for the general case where parameters are high dimensional,  our results reveal that the diffusion RLS do not necessarily outperform the individual RLS. On the other hand, the convergence theorem on the diffusion RM  in this paper and the relevant studies on the diffusion LMS    reflect that the ATC and CTA diffusion strategies might be very suitable for the adaptive algorithms in the form of the LMS-type.

\appendices

\appendixpage

\section*{Appendix A}\label{AppB}




\begin{proof}[Proof of Remark \ref{r1}(i)(b)]
The argument is based on the proof of Theorem \ref{ls1} from \dref{jj}--\dref{trthe0}. Considering \dref{Ee>0}, let $l$ be the smallest integer such that $\sum_{i=1}^{n}\phi_{l,i}^2\not=0$. An analogous proof of Theorem \ref{ls1} shows that for some $\mu'_i>0,i=1,\ldots,n$, $$\lim_{k\rightarrow+\infty}\Pi(k,k-l)=\textbf{1}\cdot(\mu'_1,\ldots,\mu'_n).$$
As a result,
 \begin{eqnarray}
\liminf_{k\rightarrow+\infty}\mbox{tr}(\Lambda_{k+1})
&\geq & \liminf_{k\rightarrow+\infty}\mbox{tr}(\Pi(k,k-l)\mathcal{A}L_{l}E[V_{l}V_{l}^{\tau}]L_{l}\mathcal{A}^{\tau}\Pi(k,k-l)^{\tau})\nonumber\\
&= &n\sum_{i=1}^{n}\bigg(\sum_{j=1}^{n}a_{ij}\mu'_{i}\bigg)^2P_{l+1,i}^2\phi_{l,i}^2E\varepsilon_{l,i}^2>0,\nonumber
\end{eqnarray}
and $\widetilde\Theta_{k}\stackrel{p}{\nrightarrow} 0$ follows as
proved in Theorem \ref{ls1}.
\end{proof}

\begin{proof}[Proof of Theorem \ref{lsr}]
Let
$$\mathcal{G}_{k}\triangleq\sigma\lbrace \Phi_{i}, V_{l}, 0\leq i\leq k, 0\leq l\leq k-1\rbrace,$$
then by \dref{1w} and Corollary \ref{lsl},
\begin{eqnarray*}
E[\widetilde\Theta_{k+1}^{\tau}\widetilde\Theta_{k+1}|\mathcal{G}_{k}]\leq \widetilde\Theta_{k}^{\tau}\widetilde\Theta_{k}-s\sum_{i=1}^{n}(1-(P_{k+1,i}P_{k,i}^{-1})^2)\|\widetilde\Theta_{k}\|^2+nM\sum_{i=1}^{n}P_{k+1,i}^2\phi_{k,i}^2.
\end{eqnarray*}
Since  $\widetilde\Theta_{k}^{\tau}\widetilde\Theta_{k}\in\mathcal{G}_{k}$, according to \dref{wsj} and
  \cite[Lemma 1.2.2]{ch02}, $
\lim_{k\rightarrow+\infty}\widetilde\Theta_{k+1}^{\tau}\widetilde\Theta_{k+1}
$ exists almost surely
and
\begin{eqnarray}\label{sl}
\sum_{k=0}^{+\infty}\sum_{i=1}^{n}(1-(P_{k+1,i}P_{k,i}^{-1})^2)\|\widetilde\Theta_{k}\|^2<+\infty,\quad\mbox{a.s.}.
\end{eqnarray}
Denote $\Theta_{\infty}\triangleq \lim_{k\rightarrow+\infty}\widetilde\Theta_{k}^{\tau}\widetilde\Theta_{k}$ and
\begin{eqnarray*}
\begin{array}{l}
 S\triangleq \lbrace \Theta_{\infty}\not=0\rbrace \cap \left \lbrace \lim_{k\rightarrow+\infty}\sum_{i=1}^nP_{k+1,i}^{-1}=+\infty\right\rbrace\\
 S'\triangleq\lbrace \Theta_{\infty}\not=0\rbrace \cap \left\lbrace \sum\limits_{k=0}^{+\infty}\sum\limits_{i=1}^{n}(1-(P_{k+1,i}P_{k,i}^{-1})^2)=+\infty \right\rbrace
 \end{array}.
\end{eqnarray*}
Note that by \dref{sumPphiinf},
\begin{eqnarray}
\bigg\lbrace \lim_{k\rightarrow+\infty}\sum_{i=1}^nP_{k+1,i}^{-1}=+\infty\bigg\rbrace\subset \bigg\lbrace \sum_{k=0}^{+\infty}\sum_{i=1}^{n}(1-(P_{k+1,i}P_{k,i}^{-1})^2)=+\infty \bigg\rbrace,\nonumber
\end{eqnarray}
then $S\subset S'$. Moreover,
\begin{eqnarray}
\sum_{k=0}^{+\infty}\sum_{i=1}^{n}(1-(P_{k+1,i}P_{k,i}^{-1})^2)\|\widetilde\Theta_{k}\|^2=+\infty\quad \mbox{on}\,\,S', \nonumber
\end{eqnarray}
which implies  $P(S)\leq P(S')=0$ by \dref{sl}.
\end{proof}

\section*{Appendix B}\label{AppA}

\begin{proof}[Proof of Lemma \ref{12}]
The first step is to seek a pair $(v_1,v_2)$ that
\begin{eqnarray}\label{Q3>L}
|Q_3(v_1,v_2)[m(l-1)+1,1]|>L.
\end{eqnarray}
To this end, denote $D(v_1)\triangleq(\mathcal{A}\otimes I_{m})Q_1(v_1)C.$
In the later discussion, we suppress $v_1$ in $D(v_1)$ for brevity.
Calculate
\begin{eqnarray}
&&Q_3(v_1,v_2)[m(l-1)+1,1]\nonumber\\
&=&((\mathcal{A}\otimes I_{m})(Q_2(v_1,v_2))D)[m(l-1)+1,1]\nonumber\\
&=&a_{lj^{*}}(1,0,\ldots,0)B_2B_1^{-1}\cdot(D[m(j^{*}-1)+1,1],\ldots,D[mj^{*},1])^{\tau}+\sum_{i\not=j^{*}}a_{li}D[m(i-1)+1,1]\nonumber\\
&=&a_{lj^{*}}(1,0,\ldots,0)\left(I_m-\frac{B_1v_2v_2^{\tau}}{1+v_2^{\tau}B_1v_2}\right)\cdot(D[m(j^{*}-1)+1,1],\ldots,D[mj^{*},1])^{\tau}\nonumber\\
&&+\sum_{i\not=j^{*}}a_{li}D[m(i-1),1]\nonumber\\
&=&-a_{lj^{*}}\frac{v_2^{\tau}D_1B_1v_2}{1+v_2^{\tau}B_1v_2}+\sum_{i=1}^{n}a_{li}D[m(i-1)+1,1]\nonumber,
\end{eqnarray}
where $D_1\in\mathbb{R}^{m\times m}$ is defined by
\begin{eqnarray}
D_1\triangleq\begin{bmatrix}
D[m(j^{*}-1)+1,1] & 0& \ldots & 0\\
\vdots & \vdots & \ddots & \vdots\\
D[mj^{*},1]& 0 & \ldots &0\\
\end{bmatrix}.\nonumber
\end{eqnarray}
Similarly, for all $i=1,\ldots,n$,
\begin{eqnarray}\label{d1}
D[m(i-1)+1,1]&=&-a_{ij^{*}}\frac{v_1^{\tau}C_1Bv_1}{1+v_1^{\tau}Bv_1}+\sum_{k=1}^{n}a_{ik}C[m(k-1)+1,1],\qquad
\end{eqnarray}
where $C_1\in\mathbb{R}^{m\times m}$ is defined by
\begin{eqnarray}
C_1\triangleq\begin{bmatrix}
C[m(j^{*}-1)+1,1] & 0& \ldots & 0\\
\vdots & \vdots & \ddots & \vdots\\
C[mj^{*},1]& 0 & \ldots &0\\
\end{bmatrix}.\nonumber
\end{eqnarray}

Now, write $v_i=r_iz_i,$  where $r_i>0$ and $|z_i|=1$, $i=1,2$.
Since for any $r_1>0$,
 $$a_{ij^{*}}\frac{|v_1^{\tau}C_1Bv_1|}{1+v_1^{\tau}Bv_1}=a_{ij^{*}}\frac{|z_1^{\tau}C_1Bz_1|}{r_1^{-2}+z_1^{\tau}Bz_1}< a_{ij^{*}}\frac{|z_1^{\tau}C_1Bz_1|}{z_1^{\tau}Bz_1},$$ 
it is trivial that
 $$a_{ij^{*}}\frac{|z_1^{\tau}C_1Bz_1|}{z_1^{\tau}Bz_1}\leq a_{ij^{*}} \lambda_{max}(B^{-1})\|B\|  \|C\|_1.$$
Then,  by \dref{d1},  for all $i=1,\ldots,n$, 
\begin{eqnarray}
|D[m(i-1)+1,1]|\leq (1+\lambda_{max}(B^{-1})\|B\|)  \|C\|_1,\nonumber
\end{eqnarray}
which infers
\begin{eqnarray}
Q_3(v_1,v_2)[m(l-1)+1,1]<-a_{lj^{*}}\frac{v_2^{\tau}D_1B_1v_2}{1+v_2^{\tau}B_1v_2}+(1+\lambda_{max}(B^{-1})\|B\|)  \|C\|_1.\nonumber
\end{eqnarray}

Next, for any $L>0$, denote
$$c\triangleq L\cdot a_{lj^{*}}^{-1}+a_{lj^{*}}^{-1}(1+\lambda_{max}(B^{-1})\|B\|)  \|C\|_1.$$
If we could find  a $v_1$ such that $D\not\in\mathcal{P}_0$ and
 $$K\triangleq{2cB_1-(D_1B_1+B_1D_1^{\tau})}$$ is not semi-positive definite, 
then there is a $v'_2$ such that for any $v_2$ in some sufficiently small neighbourhood of $v_2'$,
$$z_2^{\tau}(D_1-cI_m)B_1z_2>\frac{c}{r_2^2},$$
 which can deduce \dref{Q3>L}.
So, according to  Corollary \ref{four1}, there exists a  $v_2$ in this neighbourhood fulfilling both \dref{Q3>L} and $$Q_3(v_1,v_2)=(\mathcal{A}\otimes I_{m}) Q_2(v_1,v_2)D\not\in\mathcal{P}_d.$$

To construct the desired $v_1$, compute the  leading principal minor  of order $2$      of  $K$   by
\begin{eqnarray}
&&K[1,1]K[2,2]-K^2[1,2]\nonumber\\
&=&4(cB_1[1,1]-D_1[1,1]B_1[1,1])(cB_1[2,2]-D_1[2,1]B_1[1,2])\nonumber\\
&&-(2cB[1,2]-D_1[2,1]B_1[1,1]-D_1[1,1]B_1[1,2])^2\nonumber\\
&=&4c(c-D_1[1,1])(B_1[1,1]B_1[2,2]-B_1^2[1,2])-(D_1[1,1]B_1[1,2]-D_1[2,1]B_1[1,1])^2.\nonumber
\end{eqnarray}
Let  $z_1=(q_1,q_2,0,\ldots,0)^{\tau}$, where $q_1,q_2$ are two real numbers satisfying $q_1^2+q_2^2=1$ and $q_2\not=0$. Then,
\begin{eqnarray}\label{K12<0}
K[1,1]K[2,2]-K^2[1,2]<0
\end{eqnarray}
is equivalent to
\begin{eqnarray}\label{key}
&&4c(c-D_1[1,1])\left((B_1^{-1})^{\ast}[2,2]-\frac{((B_1^{-1})^{\ast}[1,2])^2}{(B_1^{-1})^{\ast}[1,1]}\right)\nonumber\\
&<&\frac{(D_1[1,1](B_1^{-1})^{\ast}[1,2]-D_1[2,1](B_1^{-1})^{\ast}[1,1])^2}{(B_1^{-1})^{\ast}[1,1]}.
\end{eqnarray}

Calculating the adjoint matrix of $B_1^{-1}$ shows that  there exist two constants $M_1,M_2>0$ depending on $B$ such that
$
|l_i|<M_2$ for $ i=1,2,3,
$
where
\begin{eqnarray}
\left\{
\begin{array}{l}
l_1\triangleq(B_1^{-1})^{\ast}[1,1]-r_1^2q_2^2M_1\\
l_2\triangleq(B_1^{-1})^{\ast}[1,2]+r_1^2q_1q_2M_1\\
l_3\triangleq(B_1^{-1})^{\ast}[2,2]-r_1^2q_1^2M_1
\end{array}.
\right.\nonumber
\end{eqnarray}
Therefore,
\begin{eqnarray}
\left|(B_1^{-1})^{\ast}[2,2]-\frac{((B_1^{-1})^{\ast}[1,2])^2}{(B_1^{-1})^{\ast}[1,1]}\right|
&=&\left|\frac{r_1^2(l_3q_2^2M_1+l_1q_1^2M_1+2q_1q_2M_1l_2)+l_1l_3-l_2^2}{r_1^2q_2^2M_1+l_1}\right|\nonumber\\
&\leq &\frac{2r_1^2M_1M_2(q_1^2+q_2^2)+2M_2^2}{\left|r_1^2q_2^2M_1+l_1\right|},\nonumber
\end{eqnarray}
which yields
\begin{eqnarray}\label{key2}
\limsup_{r_1\rightarrow+\infty}\left|(B_1^{-1})^{\ast}[2,2]-\frac{((B_1^{-1})^{\ast}[1,2])^2}{(B_1^{-1})^{\ast}[1,1]}\right|\leq 1+\frac{2M_2}{q_2^2}.
\end{eqnarray}
In order to estimate the right hand side of \dref{key}, we define two functions $H_1(\cdot)$ and $H_2(\cdot)$ by
\begin{eqnarray}
H_1\left(\frac{q_1}{q_2}\right)&\triangleq &\lim_{r_1\rightarrow+\infty}D_1[1,1]=\sum_{k=1}^{n}a_{j^{*}k}C[m(k-1)+1,1]-a_{j^{*}j^{*}}\frac{z_1^{\tau}C_1Bz_1}{z_1^{\tau}Bz_1}\nonumber
\end{eqnarray}
and
\begin{eqnarray}
H_2\left(\frac{q_1}{q_2}\right)&\triangleq &\lim_{r_1\rightarrow+\infty}D_1[2,1]=\sum_{k=1}^{n}a_{j^{*}k}C[m(k-1)+2,1]-a_{j^{*}j^{*}}\frac{z_1^{\tau}C_2Bz_1}{z_1^{\tau}Bz_1}\nonumber,
\end{eqnarray}
where $C_2\in\mathbb{R}^{m\times m}$ satisfies
\begin{eqnarray}
C_2=\begin{bmatrix}
0 &C[m(j^{*}-1)+1,1] & 0& \ldots & 0\\
\vdots &\vdots & \vdots & \ddots & \vdots\\
0 &C[mj^{*},1]& 0 & \ldots &0\\
\end{bmatrix}.\nonumber
\end{eqnarray}
As a result,
\begin{eqnarray}\label{rhsK}
&&\frac{1}{M_1r_1^2}\frac{(D_1[1,1](B_1^{-1})^{\ast}[1,2]-D_1[2,1](B_1^{-1})^{\ast}[1,1])^2}{(B_1^{-1})^{\ast}[1,1]}\nonumber\\
&=&\frac{(D_1[1,1](l_2-r_1^2q_1q_2M_1)-D_1[2,1](r_1^2q_2^2M_1+l_1))^2}{M_1r_1^2(r_1^2q_2^2M_1+l_1)},\nonumber\\
&\rightarrow &q_2^2\left(H_1\left(\frac{q_1}{q_2}\right)\frac{q_1}{q_2}+H_2\left(\frac{q_1}{q_2}\right)\right)^2
\end{eqnarray}
as $r_1\rightarrow+\infty$. Therefore, if
\begin{eqnarray}\label{b11}
H_1\left(\frac{q_1}{q_2}\right)\frac{q_1}{q_2}+H_2\left(\frac{q_1}{q_2}\right)\not=0,
\end{eqnarray}
then \dref{K12<0} will follow directly from \dref{key2} and \dref{rhsK}  by letting  $r_1>N(q_1,q_2)$ for some sufficiently large number $N(z_1)$.

So, the remainder is  to show that there is a $x\in\mathbb{R}$ such that
\begin{eqnarray}\label{H12x}
H_1(x)x+H_2(x)\not=0,
\end{eqnarray}
which is equivalent to
\begin{eqnarray}
&&-a_{j^{*}j^{*}}\frac{(x,1,0,\ldots,0)(xC_1+C_2)B(x,1,0,\ldots,0)^{\tau}}{(x,1,0,\ldots,0)B(x,1,0,\ldots,0)^{\tau}}+x\sum_{k=1}^{n}a_{j^{*}k}C[m(k-1)+1,1]\nonumber\\
&&+\sum_{k=1}^{n}a_{j^{*}k}C[m(k-1)+2,1]\not=0.\nonumber
\end{eqnarray}
If \dref{H12x} fails, then the coefficient of $x^3$ of
\begin{eqnarray}
(x,1,0,\ldots,0)B(x,1,0,\ldots,0)^{\tau}(H_1(x)x+H_2(x))\nonumber
\end{eqnarray}
is
\begin{eqnarray}
B[1,1]\cdot\sum_{k\not=j^{*}}a_{j^{*}k}C[m(k-1)+1,1]=0,\nonumber
\end{eqnarray}
which contradicts to $C\not\in\mathcal{P}_1$. So,  \dref{b11} holds  if $\frac{q_1}{q_2}=x$.

We now can conclude that  all $v_1=r_1 (q_1,q_2,0,\ldots,0)^{\tau}$
 with $q_1^2+q_2^2=1$, $q_2\not=0$, $\frac{q_1}{q_2}=x$ and $r_1>N(q_1,q_2)$ will result in \dref{K12<0}.
  Note that  $C\not\in\mathcal{P}_1$, by Corollary \ref{four1} again,   there always exists some  $v_1$ fulfilling both  $D\not\in\mathcal{P}_0$ and   \dref{K12<0},
 which means $K$ cannot be a semi-positive definite matrix.
\end{proof}

\section*{Appendix C}\label{AppC}

In this appendix, we prove Theorem \ref{lms} and Remark \ref{lmsr}(ii).

\begin{proof}[Proof of Theorem \ref{lms}] (i) We  verify
$$
\|\widetilde\Theta_{k}\|\stackrel{a.s.}{\longrightarrow} 0\,\,\,\mbox{and}\,\,\, \|\widetilde\Theta_{k}\|\stackrel{L_2}{\longrightarrow} 0\quad \mbox{as} \,\,k\rightarrow+\infty
$$
separately under Assumptions A3' and A3''.\\
\emph{Case 1:}
Consider the case where Assumption A3' holds.
Since
\begin{eqnarray}\label{gs}
\widetilde\Theta_{k+1}=(\mathcal{A}\otimes I_{m})(I_{mn}-F_{k})\widetilde\Theta_{k}+(\mathcal{A}\otimes I_{m})L_{k}V_{k},
\end{eqnarray}
by denoting $\mathcal{B}=\mathcal{A}^\tau\mathcal{A}$,  Assumption A2' shows
\begin{eqnarray}\label{h1}
&&E\widetilde\Theta_{k+1}^{\tau}\widetilde\Theta_{k+1}\nonumber\\
&=&E\left[\widetilde\Theta_{k}^{\tau}(I_{mn}-F_{k})(\mathcal{B}\otimes I_{m})(I_{mn}-F_{k})\widetilde\Theta_{k}\right]+E[V_{k}^{\tau}L_k^{\tau}(\mathcal{B}\otimes I_{m})L_{k}V_{k}]\nonumber\\
&\leq &E\left[\widetilde\Theta_{k}^{\tau}(I_{mn}-F_{k})(\mathcal{B}\otimes I_{m})(I_{mn}-F_{k})\widetilde\Theta_{k}\right]+\frac{nM}{(k+1)^{2\beta}}.
\end{eqnarray}
Note that $\widetilde\Theta_{k}\in \mathcal{F}_{k-1}$,  by Lemma \ref{gjb2} with $h=1$,
\begin{eqnarray}\label{1wb}
&&E\left[\widetilde\Theta_{k}^{\tau}(I_{mn}-F_{k})(\mathcal{B}\otimes I_{m})(I_{mn}-F_{k})\widetilde\Theta_{k}\right]\nonumber\\
&=& E\big[E\big[\widetilde\Theta_{k}^{\tau}(I_{mn}-F_{k})(\mathcal{B}\otimes I_{m})\cdot(I_{mn}-F_{k})\widetilde\Theta_{k}|\mathcal{F}_{k-1}\big]\big]\nonumber\\
&=&E\big[\widetilde\Theta_{k}^{\tau}E\big[(I_{mn}-F_{k})(\mathcal{B}\otimes I_{m})\cdot(I_{mn}-F_{k})\big|\mathcal{F}_{k-1}\big]\widetilde\Theta_{k}]\nonumber\\
&\leq & E\bigg[\widetilde\Theta_{k}^{\tau}\widetilde\Theta_{k}\bigg(1-\frac{s}{(k+1)^{\beta}}\lambda_{min}\bigg(\sum_{i=1}^{n}E\bigg[\frac{2\phi_{k,i}\phi_{k,i}^{\tau}}{1+\|\phi_{k,i}\|^2}\nonumber-\frac{1}{(k+1)^{\beta}}\frac{(\phi_{k,i}\phi_{k,i}^{\tau})^2}{(1+\|\phi_{k,i}\|^2)^2}\bigg|\mathcal{F}_{k-1}\bigg]\bigg)\bigg)\bigg]\nonumber\\
&\leq & E\bigg[\widetilde\Theta_{k}^{\tau}\widetilde\Theta_{k}\bigg(1-\frac{s}{(k+1)^{\beta}} \cdot \lambda_{min}\bigg(\sum_{i=1}^{n}E\bigg[\frac{\phi_{k,i}\phi_{k,i}^{\tau}}{1+\|\phi_{k,i}\|^2}\bigg|\mathcal{F}_{k-1}\bigg]\bigg)\bigg)\bigg]     .
\end{eqnarray}
For $\a$ and $c$  defined in Assumption A3',  \dref{h1}--\dref{1wb} yield 
\begin{eqnarray}
E\widetilde\Theta_{k+1}^{\tau}\widetilde\Theta_{k+1}\leq \left(1-\frac{scn}{(k+1)^{\alpha+\beta}}\right)E\widetilde\Theta_{k}^{\tau}\widetilde\Theta_{k}+\frac{nM}{(k+1)^{2\beta}}.\nonumber
\end{eqnarray}
So,  \cite[Lemma 4.2]{fa1967} implies
\begin{eqnarray}\label{sjx}
\limsup_{k\rightarrow+\infty}k^{\beta-\alpha}E\widetilde\Theta_{k}^{\tau}\widetilde\Theta_{k}\leq\frac{M}{sc}.
\end{eqnarray}

Now, we prove  the strong consistency. 
Since $\widetilde\Theta_{k}^{\tau}\widetilde\Theta_{k}\in\mathcal{F}_{k}^{'}\triangleq \sigma\lbrace \Phi_{j},V_{l},0\leq j\leq k, 0\leq l\leq k-1\rbrace$, similar to \dref{h1}--\dref{1wb},
\dref{gs} infers
\begin{eqnarray}\label{tj}
E\left[\widetilde\Theta_{k+1}^{\tau}\widetilde\Theta_{k+1}\big|\mathcal{F}_{k}^{'}\right]\leq \widetilde\Theta_{k}^{\tau}\widetilde\Theta_{k}+\frac{nM}{(k+1)^{2\beta}}.
\end{eqnarray}
By using Lemma \ref{ty}, \dref{tj} immediately yields
\begin{eqnarray}
\lim_{k\rightarrow+\infty}\widetilde\Theta_{k+1}^{\tau}\widetilde\Theta_{k+1}=0,\quad \mbox{a.s.}.\nonumber
\end{eqnarray}

\emph{Case 2:} Let Assumption A3'' hold.
Denote
 \begin{eqnarray*}
\left\{
\begin{array}{l}
\Gamma_{k}\triangleq(\mathcal{A}\otimes I_{m})L_{k}V_{k}V_{k}^{\tau}L_{k}^{\tau}(\mathcal{A}^{\tau}\otimes I_{m})\\
\Pi(k,i)\triangleq\prod_{j=k}^{k-i+1}(\mathcal{A}\otimes I_{m})(I_{mn}-F_{k})
\end{array}.
\right.
\end{eqnarray*}
Then,  \dref{gs} together with  Assumption A3''(ii) deduces
\begin{eqnarray}\label{dg}
E\widetilde\Theta_{k+h}\widetilde\Theta_{k+h}^{\tau}
&=&E[\Pi(k+h-1,h)\widetilde\Theta_{k}\widetilde\Theta_{k}^{\tau}\Pi(k+h-1,h)^{\tau}]
\nonumber\\
&&+\sum_{i=0}^{h-1}E[\Pi(k+h-1,i)\Gamma_{k+h-i}\Pi(k+h-1,i)^{\tau}],\nonumber
\end{eqnarray}
and hence
\begin{eqnarray*}\label{hj}
&&E\widetilde\Theta_{k+h}^{\tau}\widetilde\Theta_{k+h}=\mbox{tr}(E\widetilde\Theta_{k+h}\widetilde\Theta_{k+h}^{\tau})\nonumber\\
&=&E[\widetilde\Theta_{k}^{\tau}\Pi(k+h-1,h)^{\tau}\Pi(k+h-1,h)\widetilde\Theta_{k}]\nonumber\\
&&+\sum_{i=0}^{h-1}\mbox{tr}(E[\Pi(k+h-1,i)\Gamma_{k+h-i-1}\Pi(k+h-1,i)^{\tau}])\nonumber\\
&\leq &E[\widetilde\Theta_{k}^{\tau}\Pi(k+h-1,h)^{\tau}\Pi(k+h-1,h)\widetilde\Theta_{k}]
\nonumber\\
&&+\sum_{i=0}^{h-1}E\bigg[\prod_{j=k+h-1}^{k+h-i} \lambda_{max}((I_{mn}-F_{j})(\mathcal{B}\otimes I_{m})(I_{mn}-F_{j}))\nonumber\\
&&\cdot \lambda_{max}(\mathcal{B}\otimes I_{m})\lambda_{max}(L_{k+h-i-1}^{\tau}L_{k+h-i-1})\cdot V_{k+h-i-1}^{\tau}V_{k+h-i-1}\bigg]\nonumber\\
&\leq & E[\widetilde\Theta_{k}^{\tau}\Pi(k+h-1,h)^{\tau}\Pi(k+h-1,h)\widetilde\Theta_{k}]
+\frac{hnM}{(k+1)^{2\beta}}.
\end{eqnarray*}
Similar to \dref{1wb}--\dref{tj}, by applying Lemma \ref{gjb2} and   Assumption A3''(i), one has
\begin{eqnarray*}\label{yj}
E[\widetilde\Theta_{k}^{\tau}\Pi(k+h-1,h)^{\tau}\Pi(k+h-1,h)\widetilde\Theta_{k}]\leq \left(1-\frac{shnc}{(k+1)^{\alpha+\beta}}\right)E\widetilde\Theta_{k}^{\tau}\widetilde\Theta_{k},
\end{eqnarray*}
and finally can obtain
 \begin{eqnarray*}
\left\{
\begin{array}{l}
E\widetilde\Theta_{k+h}^{\tau}\widetilde\Theta_{k+h}\leq \left(1-\frac{shnc}{(k+1)^{\alpha+\beta}}\right)E\widetilde\Theta_{k}^{\tau}\widetilde\Theta_{k}+\frac{hnM}{(k+1)^{2\beta}}\\
E[\widetilde\Theta_{k+h}^{\tau}\widetilde\Theta_{k+h}|\mathcal{G}_{k}]\leq \widetilde\Theta_{k}^{\tau}\widetilde\Theta_{k}+\frac{hnM}{(k+1)^{2\beta}}
\end{array}.
\right.
\end{eqnarray*}
So, as the arguments for statement (i),    given any $k\geq 0$,
 \begin{eqnarray}\label{hconv}
\left\{
\begin{array}{l}
\limsup_{j\rightarrow+\infty}(k+hj)^{\beta-\alpha}E\widetilde\Theta_{k+hj}^{\tau}\widetilde\Theta_{k+hj}\leq \frac{M}{sc}\\
\lim_{j\rightarrow+\infty}\widetilde\Theta_{k+jh}^{\tau}\widetilde\Theta_{k+jh}=0,\quad \mbox{a.s.}
\end{array}.
\right.
\end{eqnarray}
The result is thus proved by taking $k=0,\ldots,h-1$.

(ii) The mean-square convergence rate has already been derived by \dref{sjx} and \dref{hconv}. The rest part is devoted to computing
 the convergence rate of the strong consistency under Assumption A3'. A similar analysis will lead to the same conclusion under Assumption A3''.

For every $\varepsilon\in(0,\beta-\alpha)$, we first use an induction method to prove that for all $j\in [1,l+1]$,
\begin{eqnarray}\label{jkk}
\lim_{k\rightarrow+\infty}k^{j(\beta-\alpha-\varepsilon)}E(\widetilde\Theta_{k}^{\tau}\widetilde\Theta_{k})^j=0.
\end{eqnarray}
Since \dref{jkk} is obviously true for $j=1$ by \dref{sjx}, 
 we assume that \dref{jkk} holds for all $j\leq k_0$ with some $k_0\in [1,l]$. Now, check \dref{jkk} for $j=k_0+1$.
Calculate
\begin{eqnarray}
\widetilde\Theta_{k+1}^{\tau}\widetilde\Theta_{k+1}&=&\widetilde\Theta_{k}^{\tau}(I_{mn}-F_{k})(\mathcal{B}\otimes I_{m})(I_{mn}-F_{k})\widetilde\Theta_{k}+V_{k}^{\tau}L_{k}(\mathcal{B}\otimes I_{m})L_{k}V_{k}\nonumber\\
&&+2\widetilde\Theta_{k}^{\tau}(I_{mn}-F_{k})(\mathcal{B}\otimes I_{m})L_{k}V_{k}\nonumber\\
&\triangleq & H_{k,1}+H_{k,2}+2H_{k,3},\nonumber
\end{eqnarray}
therefore,
\begin{eqnarray}
(\widetilde\Theta_{k+1}^{\tau}\widetilde\Theta_{k+1})^{k_0+1}=\sum_{i_1+i_2+i_3=k_0+1}2^{i_3}C_{k_0+1}^{i_1}C_{k_0+1}^{i_2}C_{k_0+1}^{i_3}H_{k,1}^{i_1}H_{k,2}^{i_2}H_{k,3}^{i_3}.\nonumber
\end{eqnarray}
We estimate $E(\widetilde\Theta_{k+1}^{\tau}\widetilde\Theta_{k+1})^{k_0+1}$ by considering the following three cases.
\\
\emph{Case 1:} $i_2+\frac{i_3}{2}\geq 1$. Then,  $i_1+\frac{i_3}{2}\leq k_0\leq l$, and by the induction hypothesis,
\begin{eqnarray*}
&&\limsup_{k\rightarrow+\infty}k^{\left(i_1+\frac{i_3}{2}\right)(\beta-\alpha-\varepsilon)}E(\widetilde\Theta_{k}^{\tau}\widetilde\Theta_{k})^{i_1+\frac{i_3}{2}}\leq \limsup_{k\rightarrow+\infty}k^{\left(i_1+\frac{i_3}{2}\right)(\beta-\alpha-\varepsilon)}(E(\widetilde\Theta_{k}^{\tau}\widetilde\Theta_{k})^{k_0})^{\frac{i_1+\frac{i_3}{2}}{k_0}}=0.\nonumber
\end{eqnarray*}
Therefore, \dref{EVl} shows
\begin{eqnarray*}
\left|E\left[H_{k,1}^{i_1}H_{k,2}^{i_2}H_{k,3}^{i_3}\right]\right|
&\leq &\frac{E[(\widetilde\Theta_{k}^{\tau}\widetilde\Theta_{k})^{i_1+\frac{i_3}{2}}\cdot E[(V_{k}^{\tau}V_{k})^{i_2+\frac{i_3}{2}}|\Phi_{j}, 0\leq j\leq k]]}{(k+1)^{2\beta i_2+\beta i_3}}\nonumber\\
&=& o\left((k+1)^{-\left(i_1+\frac{i_3}{2}\right)(\beta-\alpha-\varepsilon)-2\beta i_2-\beta i_3}\right)\nonumber\\
&=& o\left((k+1)^{-\left(k_0+1\right)(\beta-\alpha-\varepsilon)-\alpha-\beta-\varepsilon}\right).
\end{eqnarray*}
\emph{Case 2:} $i_2=0$, $i_3=1$. Since
\begin{eqnarray}
E\left[V_k|\Phi_{j}, V_i, 0\leq j\leq k, 0\leq i\leq k-1\right]=\textbf{0},\nonumber
\end{eqnarray}
it immediately follows that
\begin{eqnarray*}
&&E\left[H_{k,1}^{i_1}H_{k,2}^{i_2}H_{k,3}^{i_3}\right]=E\left[H_{k,1}^{k_0}H_{k,3}\right]\nonumber\\
&=&E\left[E\left[H_{k,1}^{k_0}H_{k,3}\Big|\Phi_{j}, V_i, 0\leq j\leq k, 0\leq i\leq k-1\right]\right]=0.
\end{eqnarray*}
\emph{Case 3:} $i_2=i_3=0$. Similar to \dref{1wb},
\begin{eqnarray}
E\left[H_{k,1}^{i_1}H_{k,2}^{i_2}H_{k,3}^{i_3}\right]&=& E\left[H_{k,1}^{k_0+1}\right]\leq E\left[(\widetilde\Theta_{k}^{\tau}\widetilde\Theta_{k})^{k_0}H_{k,1}\right] \nonumber\\
&=&E\left[(\widetilde\Theta_{k}^{\tau}\widetilde\Theta_{k})^{k_0}E\left[H_{k,1}|\mathcal{F}_{k-1}\right]\right]\leq  \left(1-\frac{scn}{(k+1)^{\alpha+\beta}}\right)E(\widetilde\Theta_{k}^{\tau}\widetilde\Theta_{k})^{k_0+1}.\nonumber
\end{eqnarray}
So, combining Cases 1--3, we deduce that as $k\rightarrow+\infty$,
\begin{eqnarray}
E(\widetilde\Theta_{k+1}^{\tau}\widetilde\Theta_{k+1})^{k_0+1}&\leq & \left(1-\frac{scn}{(k+1)^{\alpha+\beta}}\right)E(\widetilde\Theta_{k}^{\tau}\widetilde\Theta_{k})^{k_0+1}+o\left((k+1)^{-\left(k_0+1\right)(\beta-\alpha-\varepsilon)-\alpha-\beta-\varepsilon}\right).\nonumber
\end{eqnarray}
By \cite[Lemma 4.2]{fa1967} again,
\begin{eqnarray}
\lim_{k\rightarrow+\infty}k^{(k_0+1)(\beta-\alpha-\varepsilon)} E(\widetilde\Theta_{k}^{\tau}\widetilde\Theta_{k})^{k_0+1}=0,\nonumber
\end{eqnarray}
which means assertion \dref{jkk} is true for all $j\in [1,l+1]$.

Next, for any $\varepsilon\in(0,\beta-\alpha-\frac{1}{l+1})$,  select some $\varepsilon_0\in(0,\beta-\alpha-\frac{1}{l+1}-\varepsilon)$. So, $l+1>(\beta-\alpha-\varepsilon-\varepsilon_0)^{-1}.$
By \textit{Markov's inequality}, for any $\delta>0$,
\begin{eqnarray}
P\left((k+1)^{\varepsilon}\widetilde\Theta_{k}^{\tau}\widetilde\Theta_{k}>\delta\right)\leq \frac{(k+1)^{\varepsilon (l+1)}E(\widetilde\Theta_{k}^{\tau}\widetilde\Theta_{k})^{ l+1}}{\delta^{l+1}}=o((k+1)^{-(l+1)(\beta-\alpha-\varepsilon-\varepsilon_0)}),\nonumber
\end{eqnarray}
which implies $$\sum_{k=1}^{+\infty}P\left((k+1)^{\varepsilon}\widetilde\Theta_{k}^{\tau}\widetilde\Theta_{k}>\delta\right)<+\infty. $$
This together with \textit{Borel-Cantelli lemma} yields
\begin{eqnarray}\label{bc}
P\left((k+1)^{\varepsilon}\widetilde\Theta_{k}^{\tau}\widetilde\Theta_{k}>\delta,~ \mbox{i.o.}\right)=0.\nonumber
\end{eqnarray}
So, \dref{asconv} is true by noting that $\delta$ can take arbitrary values.
\end{proof}

\begin{proof}[Proof of Remark \ref{lmsr}(ii)]
At first, it is easy to verify
\begin{eqnarray}
\lambda_{max}(I_n-\mathcal{A})\leq 2-2\inf_{i\in[1,n]}a_{ii}.\nonumber
\end{eqnarray}
Taking $\varepsilon=\frac{\inf_{i\in[1,n]}a_{ii}}{2}$ in \cite[Lemmas 5.5, 5.8, 5.10]{guo2018} shows that for any $k\geq 0$, if $l$ is  sufficiently large,
\begin{eqnarray}\label{bb}
B_{j}^{\tau}B_{j}\leq (1-\varepsilon)(B_{j}^{\tau}+B_{j}),\quad j\geq 1,
\end{eqnarray}
where
\begin{eqnarray}
B_{j}\triangleq F_{k+lh+j-1}+(I_{mn}-\mathcal{A}\otimes I_{m})(I_{mn}-F_{k+lh+j-1}).\nonumber
\end{eqnarray}

Fix  $k,l\geq 0$ and define $I_{j}(A)\triangleq I_{mn}-F_{k+lh+j-1}$ in \dref{ap}.
Then,  \dref{bb} and \cite[Lemmas 5.5, 5.10]{guo2018} yield
\begin{eqnarray}
&&\lambda_{min}(E\left[I_{mn}-\psi_h^{\tau}\psi_h\big|\mathcal{F}_{k-1}\right])\nonumber\\
&\geq & \frac{\varepsilon}{(1+4(1-\varepsilon)h)^2}\cdot\frac{0.5}{(k+(l+1)h)^{\beta}}\cdot\frac{\lambda(\mathcal{G})h}{2+\lambda(\mathcal{G})}\nonumber\\
&&\cdot \lambda_{min}\left(E\left[\frac{1}{nh}\sum_{i=1}^{n}\sum_{j=k+lh}^{k+lh+h-1}\frac{\phi_{j,i}\phi_{j,i}^{\tau}}{1+\|\phi_{j,i}\|^2}\bigg|\mathcal{F}_{k-1}\right]\right)\nonumber\\
&\geq &\frac{\varepsilon}{2(1+4h)^2}\cdot\frac{(k+lh)^{\beta}}{(k+(l+1)h)^{\beta}}\cdot\frac{\lambda(\mathcal{G})}{4n}\cdot\lambda_{min}\left(E\left[\sum_{j=1}^{h}\sum_{i=1}^{n}(I_{m}-A_{j,i})\bigg|\mathcal{F}_{k-1}\right]\right)\nonumber\\
&\geq & \frac{\inf_{i\in[1,n]}a_{ii}}{32n(1+4h)^2}\lambda(\mathcal{G})\cdot\lambda_{min}\left(E\left[\sum_{j=1}^{h}\sum_{i=1}^{n}(I_{m}-A_{j,i})\bigg|\mathcal{F}_{k-1}\right]\right),\nonumber
\end{eqnarray}
where $l$ is sufficiently large.
This remark is thus proved by taking $s=\frac{\inf_{i\in[1,n]}a_{ii}}{32n(1+4h)^2}\lambda(\mathcal{G})$.
\end{proof}


\begin{thebibliography}{99}



\bibitem{ab2016}
R. Abdolee and B. Champagne, ``Diffusion LMS strategies in sensor
networks with noisy input data,'' \emph{IEEE/ACM Transactions on Networking},
vol. 24, no. 1, pp. 3--14, 2016.

\bibitem{ba2011}
A. Bertrand, M. Moonen and A. H.Sayed, ``Diffusion bias-compensated RLS estimation over adaptive networks''. \emph{IEEE Transactions on Signal Processing}, vol. 59, no. 11, pp. 5212--5224, 2011.

\bibitem{carli2008}
R. Carli, A. Chiuso, L. Schenato and S. Zampieri
, ``Distributed Kalman filtering
based on consensus strategies'', \emph{IEEE Journal on Selected Areas in
Communications}, vol. 26, no. 4, pp. 622--633, May 2008.


\bibitem{CLS2007}
F. S. Cattivelli, C. G. Lopes and A. H. Sayed, ``A diffusion RLS scheme for distributed estimation over adaptive networks'',  \emph{Proc. IEEE Workshop on Signal Process. Advances Wireless Comm. (SPAWC)}, Helsinki, Finland, pp. 1--5, June 2007.



\bibitem{Sayed08}
 F. S. Cattivelli, C. G. Lopes and A. H. Sayed, ``Diffusion recursive least-squares for distributed estimation over adaptive networks'', \emph{IEEE Trans. Signal Process.}, vol. 56, no. 5, pp. 1865--1877, 2008.






\bibitem{cg91}
H. Chen and L. Guo, {\em Identification and Stochastic Adaptive
Control}, Birkhauser: Boston, MA, 1991.

\bibitem{ch02}
H. Chen, {\em Stochastic Approximation and Its Applications},  Kluwer Academic Publishers, 2002.


\bibitem{cj2015}
J. Chen and A. H. Sayed, ``On the learning behavior of adaptive
networks part I: transient analysis'', \emph{IEEE Trans. Inf. Theory}, vol. 61,
no. 6, pp. 3487--3517, 2015.

\bibitem{cws2014}
W. S. Chen, C. Y. Wen, S. Y. Hua and C. Y. Sun, ``Distributed cooperative
adaptive identification and control for a group of continuous-time systems
with a cooperative PE condition via consensus'', \emph{IEEE Trans. Autom.
Control}, vol. 59, no. 1, pp. 91--106, 2014.

\bibitem{chung1996}
 F. R. K. Chung, ``Laplacians of graphs and Cheeger inequalities'', \emph{Combinatorica}, vol. 2, pp. 157--172, 1996.

\bibitem{hd76}
H. Drygas,  ``Weak and strong consistency of the least
squares estimators in regression models``, \emph{Zeitschrift f\"ur Wahrscheinlichkeitstheorie und Verwandte Gebiete}, vol. 34, pp. 119--127, 1976.





\bibitem{ef63}
F. Eicker,  ``Asymptotic normality and consistency of the least squares estimators for families of linear regressions'', \emph{The Annals of Mathematical Statistics} vol. 34, pp. 447--456, 1963.

\bibitem{fa1967}
V. Fabian,  ``Stochastic approximation of minima with improved asymptotic speed'', \emph{The Annals of Mathematical Statistics}, pp. 191--200, 1967.

\bibitem{gha2013}
O. N. Gharehshiran, V. Krishnamurthy and G. Yin,  ``Distributed energyaware
diffusion least mean squares: Game-theoretic learning'', \emph{IEEE
Journal of Selected Topics in Signal Processing}, vol. 7, no. 5, pp. 821--836, 2013.


\bibitem{guo1994}
L. Guo, ``Stability of recursive stochastic tracking algorithms'', \emph{SIAM
Journal on Control and Optimization}, vol. 32, pp. 1195--1225, 1994.




\bibitem{GL95}
 L. Guo and L. Ljung, ``Performance analysis of general tracking algorithms'', \emph{IEEE Trans. Autom. Control},  vol. 40, pp. 1388--1402, 1995.























\bibitem{kar2011}
S. Kar and J. M. F. Moura, ``Convergence rate analysis of distributed
gossip (linear parameter) estimation: Fundamental limits and tradeoffs'',
\emph{IEEE Journal on Selected Topics in Signal Processing}, vol. 5, no. 4, pp.
674--690, 2011.
\bibitem{kha2012}
A. Khalili, M. A. Tinati, A. Rastegarnia and J. A. Chanbers, ``Steady-state
analysis of diffusion LMS adaptive networks with noisy links'', \emph{IEEE
Trans. on Signal Processing}, vol. 60, no. 2, pp. 974--979, 2012.


\bibitem{lll2014}
Z. Liu, Y. Liu and C. Li, ``Distributed sparse recursive least-squares over networks'', \emph{IEEE Trans. Signal Processing},  vol. 62, no. 6, pp. 1386--1395, 2014.


\bibitem{LSayed08}
 C. G. Lopes and A. H. Sayed, ``Dffusion least-mean squares over adaptive networks: Formulation and performance analysis'', \emph{IEEE Trans. Signal Process.}, vol. 56, no. 7, pp. 3122--3136, 2008.

























\bibitem{mat2009}
G. Mateos, I. D. Schizas and G. B. Giannakis. ``Distributed recursive least-squares for consensus-based in-network adaptive estimation'',  \emph{IEEE Transactions on Signal Processing}, vol. 57, no. 11, pp 4583--4588, 2009.

\bibitem{mat2012}
G. Mateos and G. B. Giannakis, ``Distributed recursive least-squares: Stability and performance analysis'', \emph{IEEE Transactions on Signal Processing}, vol. 60, no. 7, pp. 3740--3754, 2012.



\bibitem{M63}
D. W. Marquardt,
``An algorithm for least-squares estimation of nonlinear parameters'',  \emph{SIAM Journal on Applied Mathematics}, vol. 11, pp. 431--441, 1963.








\bibitem{lai79}
T. L. Lai, H. Robbins and C. Z. Wei,
``Strong consistency of least squares estimates
in multiple regression II*'', \emph{Journal of Multivariate Analysis}, vol. 9, pp. 343--361, 1979.

\bibitem{lai94}
T. L. Lai, ``Asymptotic properties of nonlinear least squares estimates in stochastic regression models'', \emph{The Annals of Statistics}, vol. 22, pp. 1917--1930, 1994.

\bibitem{nos2015}
H. Nosrati, M. Shamsi, S. M. Taheri and M. H. Sedaaghi, ``Adaptive
networks under non-stationary conditions: Formulation, performance
analysis, and application'', \emph{IEEE Trans. Signal Process.}, vol. 63, no. 16,
pp. 4300--4314, 2015.

\bibitem{pig2015}
M. J. Piggott and V. Solo, ``Stability of distributed adaptive algorithms II:
Diffusion algorithms'', \emph{54th IEEE Conference on Decision and Control
(CDC)}, 2015, pp. 7428--7433.

\bibitem{sayed2006}
A. H. Sayed and C. G. Lopes, ``Distributed recursive least-squares strategies over adaptive network'', \emph{Conference on Signals, Systems and Computers}, pp. 233--237, 2006.

\bibitem{sha2016}
A. K. Sahu, S. Kar, J. M. F. Moura and H. V. Poor, ``Distributed constrained recursive nonlinear least-squares estimation: Algorithms and asymptotics'', \emph{IEEE Transactions on Signal and Information Processing over Networks}, vol. 2, no. 4, pp. 426--441, 2016.


\bibitem{St77}
 J. Sternby, ``On consistency for the method of least squares using martingale theory'', \emph{IEEE
Trans. Autom. Control} vol. 22, pp.  346--352, 1977.



\bibitem{ta2010}
N. Takahashi, I. Yamada and A. H. Sayed, ``Diffusion least-mean squares
with adaptive combiners: Formulation and performance analysis'', \emph{IEEE
Trans. on Signal Processing}, vol. 58, no. 9, pp. 4795--4810, 2010.

\bibitem{TuSayed12}
S.-Y. Tu and A. H. Sayed, ``Diffusion strategies outperform consensus strategies for distributed estimation over adaptive networks'', \emph{IEEE Trans. Signal Processing}, vol. 60, no. 12, pp. 6217--6234,  2012.








\bibitem{guo2018}
S. Y. Xie and L. Guo, ``Analysis of distributed adaptive filters based on diffusion strategies over sensor networks'', \emph{IEEE Trans. Autom. Control},  doi: 10.1109/TAC.2018.2799567.

\bibitem{guoxie2018}
S. Y. Xie and L. Guo, ``A necessary and sufficient condition for stability of LMS-based consensus adaptive filters'', \emph{Automatica}, pp. 12--19, 2018.




\end{thebibliography}
\end{document}